\documentclass{svjour3}

\usepackage[english]{babel}
\usepackage[colorlinks]{hyperref}
\hypersetup{
    linkcolor=red, 
    citecolor=blue, 
}

\usepackage{multicol}
\usepackage{graphicx}

\usepackage{amsmath}
\usepackage{amssymb}

\usepackage[latin1]{inputenc}
\usepackage{amsfonts}
\usepackage{enumitem}
\usepackage{xcolor}

\providecommand{\openbox}{\leavevmode
  \hbox to.77778em{%
  \hfil\vrule
  \vbox to.675em{\hrule width.6em\vfil\hrule}%
  \vrule\hfil}}
\makeatletter
\DeclareRobustCommand{\qed}{%
  \ifmmode
    \eqno \def\@badmath{$$}
    \let\eqno\relax \let\leqno\relax \let\veqno\relax
    \hbox{\openbox}%
  \else
    \leavevmode\unskip\penalty9999 \hbox{}\nobreak\hfill
    \quad\hbox{\openbox}%
  \fi
}
\makeatother

\newtheorem{algorithm}[theorem]{Algorithm}

\newcommand{\aff}  {\operatorname{aff} }
\newcommand{\conv}  {\operatorname{conv} }

\newcommand{\wert} {\operatorname{vert}}
\newcommand{\inte} {\operatorname{int}}
\newcommand{\Z}{\mathbb{Z}}
\newcommand{\R}{\mathbb{R}}
\newcommand{\N}{\mathbb{N}}

\graphicspath{{FIGURESQuasiminimals/}}

\begin{document}

\title{Enumeration of lattice $3$-polytopes by their number of lattice points
\thanks{Supported by grants MTM2011-22792, MTM2014-54207-P (both authors) and BES-2012-058920 (M.~Blanco) of the Spanish Ministry of Economy and Competitiveness, and by the Einstein Foundation Berlin (F.~Santos)}
}
\author{M\'onica Blanco \and Francisco Santos}

\institute
{
Departamento de Matem\'aticas, Estad\'istica y Computaci\'on,\\
Universidad de Cantabria,\\
39005 Santander, Spain\\
\email{francisco.santos@unican.es, monica.blancogomez@unican.es}
}

\maketitle

\begin{abstract}

We develop a procedure for the complete computational enumeration of lattice $3$-polytopes of width larger than one, of which there are finitely many for each given number of lattice points. 
We also implement an algorithm for doing this and enumerate those with at most eleven lattice points (there are 216\,453 of them).

In order to achieve this we prove that if $P$ is a lattice $3$-polytope of width larger than one and with at least seven lattice points then it fits in one of three categories that we call \emph{boxed}, \emph{spiked} and \emph{merged}.

Boxed polytopes have at most 11 lattice points; in particular they are finitely many, and we enumerate them completely with computer help. Spiked polytopes are infinitely many but admit a quite precise description (and enumeration). Merged polytopes are computed as a union (\emph{merging}) of two polytopes of width larger than one and strictly smaller number of lattice points.

\end{abstract}

\keywords{Lattice polytopes, unimodular equivalence, lattice points, finiteness, lattice width}
\subclass{52B10, 52B20}

\setcounter{tocdepth}{1}
\tableofcontents


\section{Introduction}

In this paper we describe an algorithm to classify all lattice $3$-polytopes of width larger than one and with a given number of lattice points, which we call its \emph{size}. We have implemented the algorithm and run it up to size eleven. Running it for larger sizes requires either more careful implementations or more computer power (or both).

In fact, as pointed to us by an anonymous referee, our results imply an independent proof of the following fact: there are only finitely many isomorphism classes of lattice $3$-polytopes of a given size $n$ and width larger than one~\cite[Corollary 22]{5points}. This proof has the advantage of not depending on previous results and being algorithmic and easily implementable.

Here we call two lattice polytopes $P$ and $Q$ \emph{isomorphic} or \emph{unimodularly equivalent} if there exists an affine unimodular transformation that maps one polytope to the other. That is, an affine map $t:\R^d\to \R^d$ with $t(\Z^d)=\Z^d$ and $t(P)=Q$. In this case we write $P\cong Q$. The \emph{width} of a lattice $d$-polytope $P$ is the minimum of $\max_{p\in P}f(p) -\min_{p\in P} f(p)$ over all choices of a (non-constant) affine integer functional $f:\R^d \to \R$. Remember that an affine functional is \emph{integer} if $f(\Z^d)\subseteq \Z$ and, in the case that $f(\Z^d)=\Z$, we say that $f$ is \emph{primitive}.

There are infinitely many lattice $3$-polytopes of width one for any fixed size $n$, but they are easy to describe: they consist of two parallel lattice polytopes of dimension $\le 2$ of sizes $n_1$ and $n_2$ at lattice distance one, with $n_1+n_2=n$.

All empty tetrahedra (that is, lattice $3$-polytopes of size $4$) have width one and were classified by White \cite{White} in terms of two parameters: their normalized volume $q$ plus an invertible element of $\Z_q$.
In~\cite{5points,6points} we gave the complete lists of lattice $3$-polytopes of sizes 5 and 6. This includes both those of width larger than one, which are finite lists, and those of width one, which are infinitely many but fully described in terms of certain integer parameters. The methods used were quite ad-hoc and based on first classifying the possible oriented matroids of the five or six lattice points and then doing a detailed case study. 

Here we take a different approach, already hinted in the last section of~\cite{5points}. 
Suppose that we know already the list of lattice $3$-polytopes of size $n-1$ and width $>1$. One can expect that \emph{most} of the polytopes of size $n$ can be obtained by ``merging'' two polytopes of size $n-1$ and width larger than one in the sense of the following definition. In it and in the rest of the paper we use the notation $P^v:=\conv(\Z^d \cap P\setminus \{v\})$ for a lattice polytope $P\subseteq \R^d$ and a vertex $v$ of it. We abbreviate $(P^v)^w$ as $P^{vw}$. Observe that $P^v$ has size $n-1$ and $P^{vw}$ has size $n-2$: 

\begin{definition} 
We say that a lattice $d$-polytope $P$ is obtained by \emph{merging $P_1$ and $P_2$} if there are vertices $v,w\in P$ such that $P_1\cong P^v$, $P_2\cong P^w$ and $P^{vw}$ is $d$-dimensional.
\end{definition}

For given $P_1$ and $P_2$, the following algorithm computes all the polytopes $P$ that can be obtained merging them. See a $2$-dimensional example in Figure~\ref{fig:merging_dim2}.

\begin{algorithm}[Merging]
\label{algorithm:merging}
\ \\
INPUT: two lattice $d$-polytopes $P_1$ and $P_2$ of size $n-1$.\\
OUTPUT: all the lattice $d$-polytopes of size $n$ obtained merging $P_1$ and $P_2$.\\
For each vertex $v_1$ of $P_1$ and $v_2$ of $P_2$:
\begin{enumerate}
\item Let $P'_1:=P_1^{v_1}\subseteq P_1$ and $P'_2:=P_2^{v_2}\subseteq P_2$.
\item Check that $P'_1$ and $P'_2$ are $d$-dimensional.
\item For each unimodular transformation $t: \R^d\to \R^d$ with $t(P'_1)=P'_2$, if the size of $P:= \conv(\{t(v_1)\}\cup P_2)=\conv(t(P_1)\cup \{v_2\})$ equals $n$, add $P$ to the output list.
(Observe that $t$ may not be unique, but there are finitely many possibilities for it). 
\end{enumerate}
\end{algorithm}

\vspace{-5ex}

\begin{figure}[h]
\centerline{\includegraphics[scale=.6]{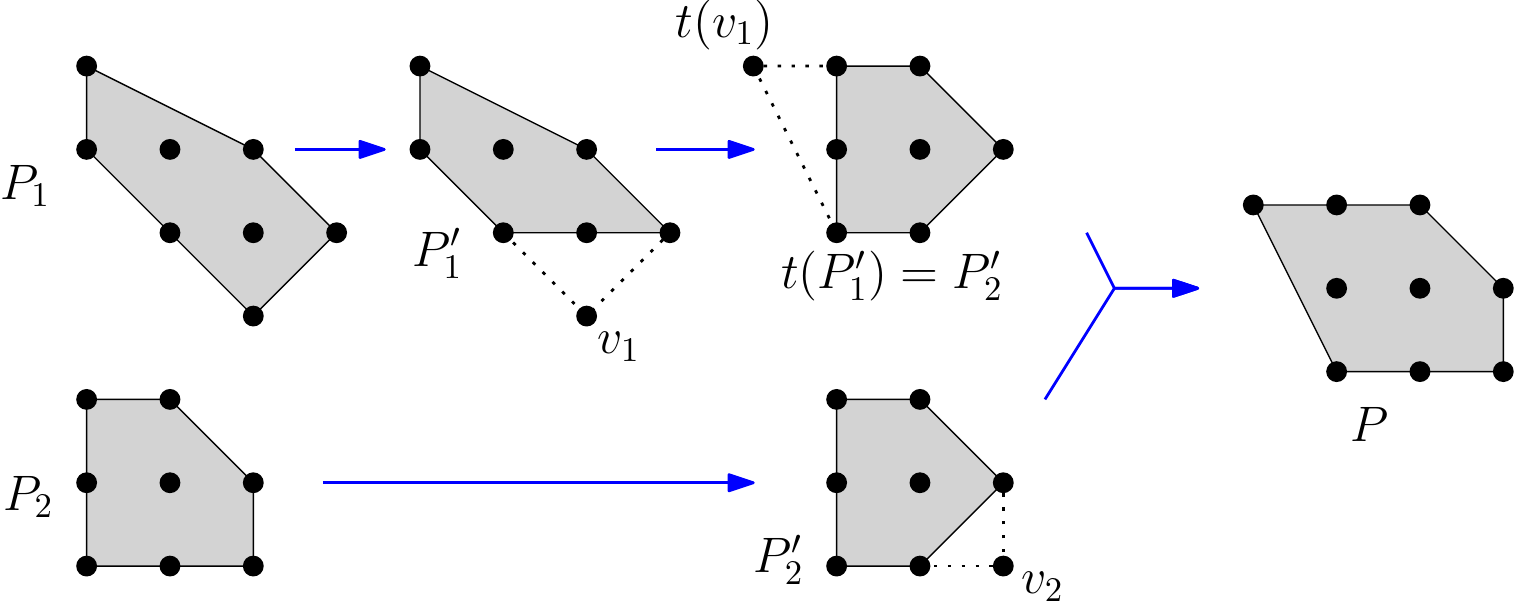}}
\caption{A lattice polygon $P$ of size $9$ constructed by merging two polygons of size $8$.}
\label{fig:merging_dim2}
\end{figure}

A brute force algorithm for step (3) is: choose an affine ordered basis $B_1$ consisting of lattice points in $P'_1$ and, for each of the at most $(d+1)!{n-2 \choose d+1}$ affine ordered bases $B_2$ consisting of lattice points in $P'_2$, consider the unique affine map $t:\R^d\to\R^d$ sending $B_1$ to $B_2$ in that order. $t$ is a unimodular equivalence if (and only if) $\det(t)=1$, $t(P'_1)=P'_2$ and it has integer coefficients. This algorithm can be made faster by first computing and comparing certain unimodular equivalence invariants of $P'_1$ and $P'_2$, most notably their volume vectors, as defined in~\cite{5points}.

\medskip

In the rest of the introduction we will concentrate on dimension $3$.
Thanks to Algorithm~\ref{algorithm:merging}, to completely enumerate lattice $3$-polytopes of a given size and width $>1$ we only need to understand (and enumerate) the lattice $3$-polytopes that are \emph{not} obtained by merging smaller polytopes of width larger than one. 

For this we introduce the following definitions.
We say that a lattice $3$-polytope $P$ of width $>1$ is:
\begin{itemize}
\item \emph{Quasi-minimal} if it has at most one vertex $v$ such that $P^v$ is still of width larger than one (see more precise phrasings in Definition~\ref{definition:minimal}).

\item \emph{Merged} if there exist lattice $3$-polytopes $P_1$ and $P_2$ of width larger than one such that $P$ is obtained merging $P_1$ and $P_2$. 
Equivalently, if $P$ has two vertices $v,w$ such that $P^v$ and $P^w$ have width larger than one and $P^{vw}$ is $3$-dimensional.

\end{itemize}

Observe that a lattice $3$-polytope $P$ of width $>1$ may have two vertices $v_1$ and $v_2$ with $P^{v_1}$ and $P^{v_2}$ of width larger than one and not be merged, because $P^{v_1v_2}$ can be $2$-dimensional. (This is excluded in the definition of merging, because merging over a $2$-dimensional intersection makes the set of transformations $t$ to be tested in Algorithm~\ref{algorithm:merging} infinite). Section~\ref{sec:finiteness-exception} is aimed at proving that there is no lattice $3$-polytope of size greater than six in which this is a problem:

\begin{theorem}[see Theorem~\ref{thm:exception}]
\label{theorem:exception_intro}
All lattice $3$-polytopes $P$ of size $n\ge 7$ and width larger than one are either quasi-minimal, or merged.
\end{theorem}

Once this is established, the algorithm for classifying all lattice $3$-polytopes of size $n>6$ and width larger than one consists simply in computing all mergings of polytopes of size $n-1$ and width $>1$ (which are assumed recursively precomputed), and adding to those the quasi-minimal ones. Computing mergings is done via Algorithm~\ref{algorithm:merging}, but the quasi-minimal polytopes need to be computed. 
\medskip

We divide quasi-minimal polytopes in two types: \emph{spiked} and \emph{boxed}, which are roughly described as having \emph{most of their lattice points lying in a lattice segment} or a \emph{rational parallelepiped}, respectively. In Section~\ref{sec:dichotomy} we explain these concepts in detail (Definitions~\ref{def:spiked} and~\ref{def:boxed}), give examples, and show that every quasi-minimal polytope is either spiked or boxed (Theorem~\ref{thm:spiked_vs_boxed}).
\smallskip

Section~\ref{sec:spiked} is devoted to the study and classification of spiked $3$-polytopes, for which the main tool is the following theorem:
\begin{theorem}[see Corollary~\ref{coro:spiked-projections}]
\label{thm:spiked_intro}
Let $P$ be a spiked quasi-minimal lattice $3$-polytope with at least $7$ lattice points. Then $P$ projects to one of the following polygons in such a way that each of the vertices in the projection has a unique element in the preimage. 
\smallskip

\centerline{
\includegraphics[scale=0.6]{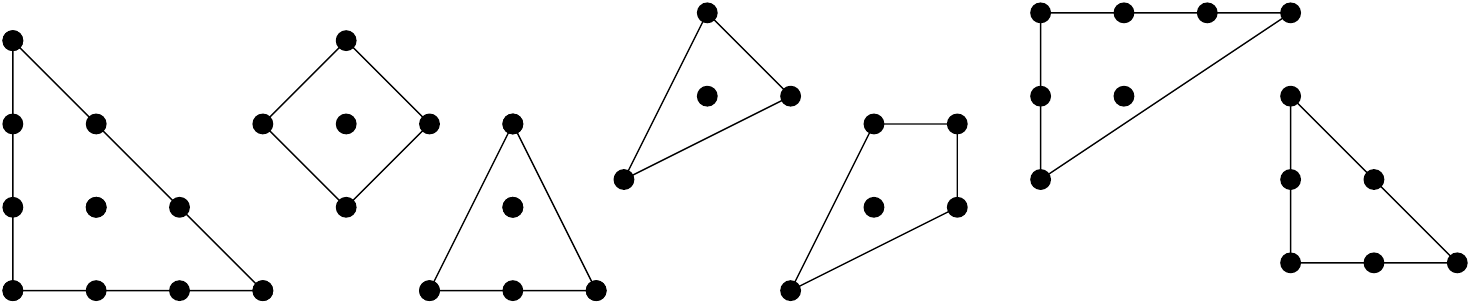}
}
\end{theorem} 

Here and in the rest of the paper when we say ``$P$ projects to'' we mean via a \emph{lattice projection}; that is, an affine map $\pi:\R^d\to \R^k$ such that $\pi(\Z^d) =\Z^k$.

There are infinitely many spiked $3$-polytopes in total but only finitely many for each size, and they are very explicitly described in Theorems~\ref{theorem:spiked-minimals} and~\ref{theorem:spiked-quasiminimals}. 
\smallskip

Boxed $3$-polytopes are classified in Section~\ref{sec:boxed}. They have at most $11$ lattice points (see Remark~\ref{rm:box_spi_def}) and they are finitely many in total, but they also have less structure so their enumeration is in fact more complicated than that of spiked ones. Their defining property is that $P$ is boxed if there is a rational parallelepiped $Q$ of width one with respect to every facet and such that at most three lattice points $v_1$, $v_2$ and $v_3$ of $P$ do not lie in $Q$ (see a more precise definition in Section~\ref{sec:dichotomy}). Two problems arise, that we solve in Section~\ref{sec:boxed}:
\begin{enumerate}
\item A priori there are many possibilities for $Q$. Lemma~\ref{lemma:cubes} shows that if $P$ has size at least seven then there are only two: the unit cube and a certain parallelepiped with four integer and four non-integer vertices.
\item A priori there are infinitely many possibilities to check for the $v_i$'s. We solve this in Theorem~\ref{thm:boxed_wrt_unit_cube}, by showing that the $v_i$'s must be at distance at most six from $Q$, which reduces the possibilities to finitely many. 
\end{enumerate}
Once this is proved, a complete enumeration of boxed $3$-polytopes is possible via a computer exhaustive search, as explained in Section~\ref{subsec:computer-routines}. Table~\ref{table-minimal} in Section~\ref{sec:result-tables} shows the numbers of quasi-minimal boxed and spiked $3$-polytopes for each size and number of vertices.
\medskip

After we have a classification of quasi-minimal $3$-polytopes we can run the algorithm. Our results are summarized as follows:

\begin{theorem}
There are $9$, $76$, $496$, $2675$, $11698$, $45035$ and $156464$ lattice $3$-polytopes of width larger than one and sizes $5$, $6$, $7$, $8$, $9$, $10$ and $11$, respectively. 
\end{theorem}

The complete lists of these polytopes are available at \url{http://personales.unican.es/santosf/3polytopes/}.

\begin{remark}
\rm
It may seem that this theorem and our algorithm make the results in~\cite{5points,6points} useless, but that is not the case. Quite the opposite, we need those results as the starting point for our algorithm, since there are several issues that make the techniques in this paper only applicable for size at least seven. 
One of them is already hinted in Theorem~\ref{theorem:exception_intro}. But most importantly, the assumption of the size being at least seven considerably simplifies the cases to be considered in the classification of spiked and boxed $3$-polytopes (Theorem~\ref{thm:projection_shape} and Lemma~\ref{lemma:cubes}). In particular, classifying quasi-minimal $3$-polytopes of size six might have been not significantly simpler than repeating the work done in~\cite{6points}. 
\end{remark}

\medskip

Finally, in Section~\ref{sec:result-tables} we give more detailed information about the output of the algorithm and some remarks that can be derived from it. In the following summary, all results are given for lattice $3$-polytopes of width larger than one and sizes between $5$ and $11$.

Tables~\ref{table:main} and~\ref{table:interior} show the numbers of polytopes of each size in terms of their numbers of vertices and of interior points.
Kasprzyk~\cite{Kasprzyk-database} and Balletti and Kasprzyk~\cite{BaKa} have enumerated all lattice $3$-polytopes with exactly one or two interior lattice points. Our results agree with theirs. Those with one interior lattice point are specially important for their connections to toric geometry and are called \emph{canonical}. If, moreover, the interior point is the unique non-vertex they are called \emph{terminal}. Table~\ref{table:canonical} shows the numbers of canonical and terminal $3$-polytopes up to size $11$. 

Table~\ref{table:width} shows the classification according to width. The maximum widths achieved by the polytopes of each size are: width $2$ for size $5$, $3$ for sizes $6$ to $9$, and $4$ for sizes $10$ and $11$.

In Section~\ref{subsec:results-volumes} we look at the volumes that arise for each size. Experimentally, we see that in each size $n\in\{5,\dots,11\}$ there is always a unique polytope that maximizes volume, and it is a tetrahedron of volume $12(n-4)+8$. This tetrahedron can be generalized to arbitrary size (Proposition~\ref{pro:biggest_polytope}) and we conjecture it to be the unique maximizer of volume for every size (Conjecture~\ref{conj:volume}).

Most of the polytopes in our output are \emph{lattice-spanning}, by which we mean that their integer points affinely span the whole integer lattice. When this is not the case, we call \emph{sublattice index} of a lattice $3$-polytope $P$ the index of the lattice spanned by $P\cap\Z^3$ as a sublattice of $\Z^3$. Experimentally we see that there is a unique polytope of sublattice index five (a terminal tetrahedron) and that the only other subindices that arise are two and three (see Table~\ref{table:nonprimitive}). In a subsequent paper~\cite{nonprimitive} we show that this is actually the case for lattice $3$-polytopes of width $>1$ and any size and characterize those that are not lattice-spanning: there are linearly many of index three and quadratically many of index two, for each size. This is derived by induction on the size of the polytopes, using the classifications of quasi-minimal $3$-polytopes and the fact that every lattice $3$-polytope of size at least seven and width larger than one is either quasi-minimal or merged.

We then look at how many of our polytopes are \emph{normal} (Section~\ref{sec:normality}). Experimentally it seems that the fraction of polytopes that are normal does not vary much with size and stays close to a 13\%. We do not know whether the same keeps happening for higher sizes. We also check that up to size 11, all polytopes in our database have a vertex that can be removed and still leave a normal polytope (whether this is always the case is a question from~\cite{Bruns_etal}).

Our studies are partially motivated by the concept of \emph{distinct pair-sums polytopes} (or \emph{dps} polytopes for short), first introduced in~\cite{ChoiLamReznick}. Dps polytopes are lattice polytopes in which all the pairwise sums $\{a+b: a,b\in P\cap \Z^d\}$ are distinct. Equivalently, they are lattice polytopes containing neither three collinear lattice points nor the vertices of a parallelogram~\protect{\cite[Lemma~1]{ChoiLamReznick}}. They coincide with the polytopes of \emph{Minkowski length} equal to one~\cite{Beckwith_etal} and, for this reason, they are called \emph{strongly indecomposable} in~\cite{Soprunovs}.
Dps polytopes of dimension $d$ have at most $2^d$ lattice points, hence our database contains the complete classification of lattice dps $3$-polytopes of width larger than one. We devote Section~\ref{subsec:results-dps} to them.
Table~\ref{table:dps} gives the number of dps polytopes for each size and number of vertices. In particular we can answer in dimension $3$ the several questions posed by Reznick~\cite{Reznick-favorite} regarding dps polytopes. We also observe that dps $d$-polytopes for $d\in\{2,3\}$ have at most $3\cdot2^{d-2}$ vertices and ask whether the same happens in higher dimensions (see Question~\ref{qn:dps_vertices}).
\medskip

Throughout the paper, $x$, $y$ and $z$ denote the coordinate functionals in $\R^3$. We sometimes denote $ab$ and $abc$ the segment $\conv \{a,b\}$ and the triangle $\conv \{a,b,c\}$, for $a,b,c\in\R^3$. We use the book \cite{Ziegler} for reference on polytopes.

\begin{remark}
After this work was completed we have learned that a full classification of strongly indecomposable $3$-polytopes is contained in the unpublished PhD thesis of J.~Whitney~\cite{Whitney}.
This classification agrees with ours.
\end{remark}

\subsection*{Acknowledgements} We thank Gabriele Balletti for sharing with us his results (\cite{BaKa}, joint work with A.~Kasprzyk) on the classification of lattice $3$-polytopes with two interior points. In particular, a discrepancy between their results and a preliminary version of ours led us to correct a mistake in a first version of Theorem~\ref{thm:projection_shape} and Corollary~\ref{coro:spiked-projections}.

We thank I.~Soprunov and J.~Soprunova for pointing us to reference~\cite{Whitney}.

We thank the anonymous referees of this paper for their thorough revision of it and the constructive comments they made.


\section{Quasi-minimal polytopes: spiked vs.~boxed}
\label{sec:finiteness}

Throughout this section, let $A \subseteq \Z^d$ be a finite set of lattice points such that $\conv(A)$ has width greater than one.
For each vertex $v$ of $\conv(A)$ we denote $A^v:=A \setminus \{v\}$. 
\begin{definition}
We say that a vertex $v$ of $\conv(A)$ is \emph{essential} if $\conv(A^v)$ has width at most one. That is, if $\conv(A^v)$ either has width one or is $(d-1)$-dimensional. 
\end{definition}

Let $\wert(A)$ be the set of all vertices of $\conv(A)$ and $ \wert^*(A) \subseteq \wert(A)$ the set of essential vertices.
We are primarily interested in the case $A=P\cap \Z^d$, for a lattice $d$-polytope $P$, in which case we use $\wert(P)$, $\wert^*(P)$ and $P^v$ for $\wert(A)$, $\wert^*(A)$ and $\conv(A^v)$, but we also need to consider more general cases.

\begin{definition}
\label{definition:minimal}
We say that a configuration $A$ of width larger than one is \emph{minimal} if all its vertices are essential, and that it is \emph{quasi-minimal} if at most one vertex is not essential. That is, if $\wert^* (A)=\wert (A)$ and $|\wert^*(A) |\ge |\wert(A)|-1$, respectively.

We say that a lattice $d$-polytope $P$ of width larger than one is \emph{minimal} or \emph{quasi-minimal} if $P\cap \Z^d$ is a minimal or quasi-minimal configuration, respectively.
\end{definition}

Notice that the number of essential vertices of a quasi-minimal or minimal $d$-polytope is at least $d$ or $d+1$, respectively. See Figure~\ref{fig:quasi_example} for an example showing a quasi-minimal $3$-polytope and its essential vertices.
\medskip

\begin{figure}[h]
\centerline{\includegraphics[scale=.7]{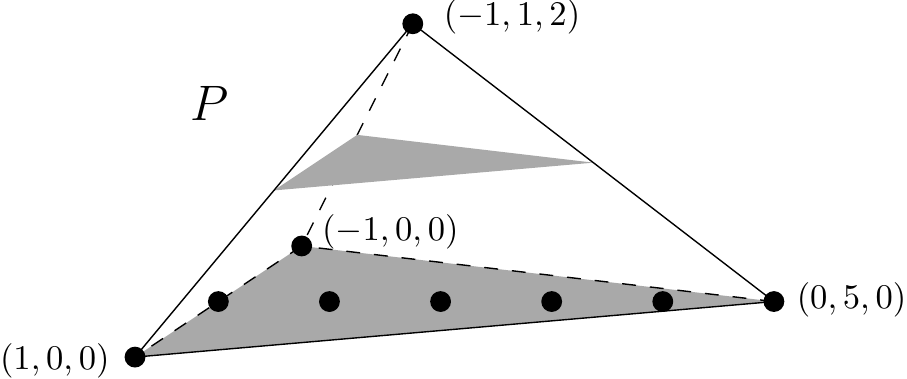}}
\caption{A quasi-minimal $3$-polytope. Black dots are lattice points of $P$. The gray triangles represent the intersection of $P$ with the planes $\{z=0,1\}$. Vertices $(-1,1,2)$, $(1,0,0)$ and $(-1,0,0)$ are essential ($P^{(-1,1,2)}$ is $2$-dimensional; $P^{(1,0,0)}$ and $P^{(-1,0,0)}$ have width one). Vertex $(0,5,0)$ is the only non-essential vertex.}
\label{fig:quasi_example}
\end{figure}

One of the main results in this paper is the complete classification of quasi-minimal $3$-polytopes. As a warm-up let us show that there are infinitely many of them in any dimension:

\begin{proposition}
For every $d\ge 2$ and $k\ge 2$ the following lattice $d$-polytope with $2(d-1) + 1$ vertices and $2(d-1) + k+1$ lattice points is quasi-minimal:
\[
\conv\{ \pm e_1,\dots, \pm e_{d-1}, ke_d\}.
\]
\end{proposition}

\begin{proof}
It has width at least two since $e_d$ is an interior lattice point in it. For every $i\in\{1,\dots,d-1\}$, removing $e_i$ or $-e_i$ gives width one with respect to the $i$-th coordinate.
\qed\end{proof}

A slight modification of this construction shows that in $d\ge 3$ there are infinitely many \emph{minimal} polytopes: 
\[
\conv\{\pm e_1,\dots, \pm e_{d-2}, - e_{d-1}, e_{d-1} + ke_d\},\; k\ge 3.
\]

In dimension 2, however, there are only four minimal polygons, as follows from the complete classification of quasi-minimal polygons that we work out in Lemma~\ref{lemma:quasi-minimal-dim2} below (see also Figure~\ref{fig:minimal-dim2}).

\vspace{-1ex}

\subsection{A dichotomy}
\label{sec:dichotomy}

Let us introduce the following two types of configurations:

\begin{definition}[Spiked configuration]
\label{def:spiked}
Let $A\subseteq \Z^d$ be a quasi-minimal configuration and let $A':=\pi(A)$, where $\pi: \R^d \to \R^{d-1}$ is a lattice projection such that:
\begin{itemize}
\item Every vertex of $\conv(A')$ has a unique preimage in $A$.
\item The projection bijects essential vertices of $A$ and $A'$.
\end{itemize}
Then, we say that $A$ is \emph{spiked with respect to $A'$}.

We say that a quasi-minimal $d$-polytope $P\subseteq \R^d$ is \emph{spiked} if $P\cap \Z^d$ is a spiked configuration.
\end{definition}

\begin{definition}[Boxed configuration]
\label{def:boxed}
Let $A\subseteq \Z^d$ be a configuration of width larger than one, and let $Q\subseteq \R^d$ be a rational $d$-dimensional parallelepiped such that:

\begin{itemize}
\item The facets of $Q$ are defined by lattice hyperplanes, with opposite facets at lattice distance one. That is, 
\[
Q=\underset{i=1}{\overset{d}{\bigcap}} f_i^{-1} ([0,1]),
\]
where the $f_i$ are affinely independent primitive integer functionals. 

\item $A\setminus Q= \{v_1, \dots,v_d\}$, with $f_i(v_j)\not \in \{0,1\}$ if and only if $i= j$. In particular, each $v_i$ is an essential vertex of $A$.
\end{itemize}
Then, we say that $A$ is \emph{boxed with respect to $Q$}.

We say that a lattice $d$-polytope $P\subseteq \R^d$ is \emph{boxed} if $P\cap \Z^d$ is a boxed configuration.
\end{definition}

\smallskip

\begin{remark}
\label{rm:box_spi_def}
The definition of spiked assumes $A$ to be quasi-minimal, but the definition of boxed does not. Observe also the following immediate consequences of the definitions: 
\begin{itemize}
\item The $(d-1)$-dimensional configuration $A'$ in the definition of spiked is automatically quasi-minimal. More precisely, it has at most as many non-essential vertices as $A$.
\item Notice that $A \cap Q \cap \Z^d\neq \emptyset$ in the definition of boxed, or otherwise $A$ only has $d$ points. Boxed configurations have at most $d+2^d$ lattice points: apart from $v_1,\dots, v_d$, only the $2^d$ vertices of $Q$ can be lattice points.
\end{itemize}
\end{remark}

\begin{example}
Observe that a quasi-minimal polytope can be both spiked and boxed. An example is $\conv\{(1,0,0), (0,1,2),(-1,0,0),(0,-1,1)\}$, whose set of lattice points is $\{(1,0,0),(0,1,2) , (-1,0,0), (0,-1,1),(0,0,0), (0,0,1)\}$. It is spiked with respect to the projection $\pi(x,y,z):=(x,y)$, and it is boxed with respect to the unit cube $[0,1]^3=\bigcap_{i=1,2,3} f_i^{-1} ([0,1])$, for $v_1=(-1,0,0)$, $v_2=(0,-1,1)$, $v_3=(0,1,2)$, and $f_1=x$, $f_2=y$, $f_3=z$.
\end{example}

\begin{example}
The second dilation of the unimodular tetrahedron ($\conv\{(0,0,0),(2,0,0)$, $(0,2,0),(0,0,2)\}$) is minimal and boxed (with respect to the parallelepiped $Q=[0,1]^3$), but not spiked.
\end{example}

\begin{theorem}
\label{thm:spiked_vs_boxed}
Every quasi-minimal configuration is spiked or boxed. 
\end{theorem}

\begin{proof}
Let $A \subseteq \Z^d$ be a quasi-minimal configuration. For each $v_i\in \wert^*(A)$ let $f_i$ be an affine primitive integer functional with $A^{v_i}\subseteq f_i^{-1}(\{0,1\})$.
Let $\hat f_i$ be the corresponding linear functional. We distinguish the following two cases:
\begin{enumerate}

\item If the set $\{\hat f_i : v_i\in \wert^*(A)\}$ linearly spans $(\R^d)^*$, then assume without loss of generality that $\hat f_1,\dots, \hat f_d$ are linearly independent. For each $i$ we have that $A^{v_i}\subseteq f_i^{-1}(\{0,1\})$ and, since $A$ has width larger than one, $v_i \not\in f_i^{-1}([0,1])$. By construction $A$ is boxed with respect to the parallelepiped
\[
\underset{i=1}{\overset{d}{\bigcap}} f_i^{-1} ([0,1]).
\]

\item If the set $\{\hat f_i : v_i\in \wert^*(A)\}$ does not linearly span $(\R^d)^*$, then there is a line where all the $\hat f_i$ are constant. That is, there exists a lattice line $r \subseteq \R^d$ such that $\hat f_i(r)=\{0\}$ for all $i$. Let $\pi: \R^d \to \R^{d}/ r $ be the quotient map, let $L:=\pi(\Z^d)\cong\Z^{d-1}$, and let $f_i':\R^d/r \to \R$ be the functional defined by $f_i(p)=f_i'(\pi(p))$ for each $p\in \R^d$. 
Since $A$ has width greater than one, so does $A':=\pi(A)$ with respect to the lattice $L$.
We claim that $A$ is spiked with respect to $A'$.

Let us first see that each vertex of $\conv(A')$ has a unique preimage in $A$. For this, let $v'$ be a vertex of $\conv(A')$ and suppose $u,v$ are two different vertices of $A$ projecting to $v'$. Then we would have $\pi(A^{u})=\pi(A^{v})=\pi(A)=A'$. But at least one of $u$ and $v$ must be an essential vertex of $A$. Say $v=v_i\in\wert^*(A)$, then $A^{v_i}$ has width one with respect to $f_i$, which is constant in $r$. This would imply $\pi(A^{v_i})=A'$ to have width one in $L$ with respect to $f_i'$, which is a contradiction.

Finally, let us prove that $\pi$ bijects essential vertices of $A$ and $A'$. Let $v'$ be a vertex of $\conv(A')$ and let $v$ be the unique vertex of $\conv(A)$ with $\pi(v)=v'$. 
Then $v\in\wert^*(A)$ if and only if $A^v$ has width one with respect to a functional $f$ with $\hat f(r)=0$. This happens if and only if the corresponding functional $f'$ gives width one to $A'\setminus\{v'\}$ in $L$, which in turn is equivalent to $v'\in\wert^*(A')$.\qed
\end{enumerate}
\end{proof}

\begin{remark}
\label{rm:clarify_boxed_vs_spiked}
It follows from the proof of Theorem~\ref{thm:spiked_vs_boxed} that a lattice $d$-polytope $P\subseteq \R^d$ is:
\begin{itemize}
\item Boxed, if there exist essential vertices $v_1,\dots,v_d$ of $P$, and linear integer functionals $f_1,\dots,f_d$ such that $\max f_i(P^{v_i}) - \min f_i(P^{v_i})\le 1$ (that is, $P^{v_i}$ has width at most one with respect to $f_i$), and such that the $f_i$ are linearly independent.
\item Spiked, if it is quasi-minimal, and there exists a direction $r\in\Z^d$ and linear integer functionals $f_v$ for each essential vertex $v$ of $P$, such that $P^v$ has width at most one with respect to $f_v$, and such that $f_v$ is constant on $r$.
\end{itemize}
\end{remark}

In dimension one the only minimal configurations are $\{0,1,2\}$ and $\{0,k\}$ for any $k\ge 2$. The only quasi-minimal ones that are not minimal are $\{0,1, k\}$, for any $k\ge 3$.
In dimension two, classifying quasi-minimal configurations is not that easy, but Theorem~\ref{thm:spiked_vs_boxed} allows us to classify quasi-minimal polygons:

\begin{lemma}
\label{lemma:quasi-minimal-dim2}
Every quasi-minimal $2$-polytope is unimodularly equivalent to one in Figure~\ref{fig:minimal-dim2}.

\begin{figure}[ht]
\begin{center}
\includegraphics[scale=0.7]{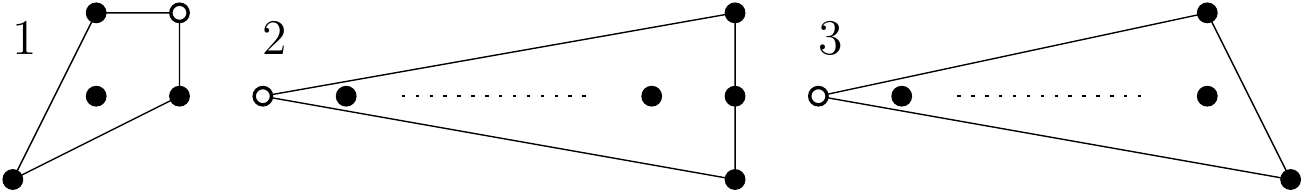}\\
\bigskip
\includegraphics[scale=0.7]{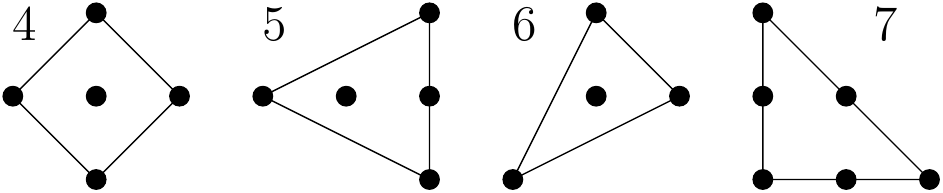}
\end{center}
\caption{The quasi-minimal lattice polygons. The black dots are lattice points and, in the non-minimal ones (top row) the white dot is the non-essential vertex. Labels 1 to 7 are used in the proof of Theorem~\ref{thm:exception}.}
\label{fig:minimal-dim2}
\end{figure}
\end{lemma}

\begin{proof}
An exhaustive search gives the list for lattice polygons with up to $6$ points.
So, for the rest of the proof let $P$ be a quasi-minimal polygon with at least $7$ lattice points. Since by Remark~\ref{rm:box_spi_def} it has to be spiked, Theorem~\ref{thm:spiked_vs_boxed} implies that $P$ projects to $\{0,1,k\}$ or $\{0,k\}$ ($k\ge 2$) with a single point in the fibers of $0$ and $k$. Hence it can only project to $\{0,1,k\}$ and must have at least five lattice points in the fiber of $1$.

If $k=2$, then $P$ is a quasi-minimal triangle like one of the two in Figure~\ref{fig:minimal-dim2} with an arbitrary number of interior lattice points. If $k\ge 3$, $\{0,1,k\}$ has a unique essential vertex, while $P$ must have at least $2$. Hence no projection will be a bijection of the essential vertices of $P$ and the essential vertices of $\{0,1,k\}$.
\qed\end{proof}

\subsection{An exception}
\label{sec:finiteness-exception}

We said in the introduction that all lattice $3$-polytopes of width $>1$ and size at least $7$ that are not quasi-minimal must be merged. Here we prove a more explicit result:

\begin{theorem}
\label{thm:exception}
For every lattice $3$-polytope $P$ of width larger than one exactly one of the following happens:
\begin{enumerate}
\item $P$ is quasi-minimal.
\item $P$ is merged.
\item $P$ has size $6$ and is equivalent to the convex hull of the columns of the following matrix

\[\left( \begin{array}{cccccc}
0  & 1 & 0 &-1 & 1 & -1\\ 
0  & 0 & 1 &-1 & 2 & -2\\ 
0  & 0 & 0 & 0 & 3 & -3\end{array} \right).\]

\end{enumerate}
\end{theorem}

The proof needs a preliminary result, which follows from the classification of lattice $3$-polytopes of sizes five and six. See~\cite{5points} and~\cite[Section 5]{6points}:

\begin{lemma}
\label{lemma:exception-height}
Let $P$ be a lattice $3$-polytope with a vertex $v$ such that $P^v$ is two-dimensional (we assume $P^v$ to be in the plane $\R^2\times\{0\}$). Then:
\begin{enumerate}
\item If $P^v$ contains either a unimodular parallelogram or the following five-point set, then $v$ is at lattice distance one from $P^v$:
\label{lemma:exception-height-a}
\[
\{(0,0), (-1,0), (-2,0), (0,1), (1,-1)\} \times \{0\}.
\]

\item If $P^v$ contains one of the following five-point sets, then either $v$ is at lattice distance one from $P^v$ or $v=(a,b,\pm 2)$ with $a \equiv 1\equiv b \mod 2$:
\label{lemma:exception-height-b}
\[
\{(-1,0), (0,0), (1,0), (0,1), (0,-1)\}\times \{0\},
\]\[
\{(0,0), (-1,0), (-2,0), (0,1), (0,-1)\}\times \{0\}.
\]
\end{enumerate}
\end{lemma}

\begin{proof}[Proof of Theorem~\ref{thm:exception}]
We first show that the three cases are mutually exclusive. Cases (1) and (2) are by definition, so we only need to check that the polytope in (3) is neither quasi-minimal nor merged. 

Let $P$ be as in part (3) of the statement. Geometrically, it is a triangular bipyramid with the origin as the common barycenter of the triangle $\conv\{(1,0,0),(0,1,0)$, $(-1,-1,0)\}$ and the segment $\conv\{(1,2,3),(-1,-2,-3))\}$. The vertices in the triangle are essential (consider the functionals $-2x+y+z$, $x-2y+z$ and $x+y-z$, corresponding to removing $(1,0,0)$, $(0,1,0)$ and $(-1,-1,0)$, respectively) but the vertices $u=(1,2,3)$ and $v=(-1,-2,-3)$ in the segment are not. This means $P$ is not quasi-minimal (it has two non-essential vertices) but it is not merged either, since $P^{uv}=\conv\{(1,0,0),(0,1,0),(-1,-1,0),(0,0,0)\}$ is two-dimensional.

So, to finish the proof it only remains to show that if $P$ is not in the conditions of (1) or (2) then it is unimodularly equivalent to the configuration in part (3).
Observe that $P$ must have size at least six, since every lattice $3$-polytope of size four has width one, which implies that every lattice $3$-polytope of size five and width larger than one is minimal. Since $P$ is neither quasi-minimal nor merged, it has at least two non-essential vertices $v_1$ and $v_2$ such that $Q:=P^{v_1v_2}$ is lower dimensional. We assume without loss of generality that $Q\subseteq \R^2\times\{0\}$. On the other hand, notice that $P^{v_i}=\conv(Q\cup \{v_j\})$ has to be $3$-dimensional and of width $>1$, since $v_i$ is not essential.
We are going to conclude that the only possibility is the configuration in (3), by proving several properties about $P$ and $Q$:

\begin{enumerate}

\item[(a)] \emph{$Q$ is $2$-dimensional}, because if $Q$ was contained in a line then $P^{v_i}$ would be at most $2$-dimensional.

\item[(b)] \emph{$v_1$ and $v_2$ both lie at lattice distance greater than 1 from (the lattice plane containing) $Q$}, since $P^{v_1}$ and $P^{v_2}$ have width greater than one.

\item[(c)] \emph{$Q$ does not contain a unimodular lattice parallelogram}, by Lemma~\ref{lemma:exception-height}\eqref{lemma:exception-height-a}.

\item[(d)] \emph{$Q$ has width larger than one}.

Suppose not, so that the lattice points of $Q$ lie in two consecutive parallel lattice lines in the plane $\aff(Q)=\{z=0\}$. If both lines have at least two lattice points of $Q$, then $Q$ contains a unimodular lattice parallelogram, which contradicts (c). If one of the lines contains only one lattice point, call it $v$, then $P^{v_1v_2v}$ is a lattice segment. 
Let $a$, $b$ and $c$ be three consecutive lattice points in $P^{v_1v_2v}$ (which exist since $P$ has size at least six). By the classification of configurations of size five (see \cite[Theorem~13]{5points}) we know that a size five polytope $\conv\{a,b,c,v,v_1\}$ with three collinear lattice points $a$, $b$ and $c$ has width one with respect to a functional that is constant in those three points. Since that functional must then be constant in the whole segment $P^{v_1v_2v}$, $P^{v_2}$ has width one as well.

\item[(e)] \emph{None of the non-essential vertices of $Q$ (if it has any) are vertices of $P$.}

Suppose otherwise, and let $w$ be a non-essential vertex of $Q$ that is a vertex of $P$. Since $Q^w$ has width larger than one and both $v_i \in P^w$ are at lattice distance greater than one from $Q^w$, $P^w$ has width larger than one as well. Thus, $w$ is also a non-essential vertex of $P$ and $P$ is merged from $P^w$ and $P^{v_i}$, for any choice of $i$, since the polytope $P^{wv_i}=\conv(Q^{w} \cup \{v_j\})$, for $\{i,j\}=\{1,2\}$, is $3$-dimensional.

\item[(f)] \emph{$Q$ is quasi-minimal (as a polygon in the lattice $\Z^2\times \{0\} \cong \Z^2$)}.

Suppose not. Then $Q$ has at least two non-essential vertices which, by part (e) are not vertices of $P$.
This implies that the segment $v_1v_2$ intersects $\aff(Q)$ in a point $v_0\in\R^2\times\{0\}$ outside $Q$ and that the two non essential vertices are in $Q_0: =\conv(Q\cup \{v_0\})$, but are not vertices of it. (Notice that if $v_1v_2$ intersects $\aff(Q)$ in a point of $Q$ then at most one vertex of $Q$ is not a vertex of $P$).
In particular, we can find two consecutive vertices $w$ and $w'$ of $Q$ such that the line containing the edge $ww'$ separates $Q$ from $v_0$ and such that $w$ and $w'$ are not vertices of $Q_0$.
Observe that the triangle $v_0ww' \subseteq Q_0$ cannot contain lattice points outside the segment $ww'$, because they would be lattice points of $Q=P^{v_1v_2}$.

To fix ideas, without loss of generality let $w=(0,0,0)$, $w'=(k,0,0)$, and $v_0=(a,b,0)$ with $k,b>0$ and $k\in \Z$. 
Consider the first lattice point $w''$ in the segment from $w$ to $w'$ (which could equal $w'$); that is, let $w''=(1,0,0)$. Since the triangle $ww''v_0$ does not contain lattice points other than $w$ and $w''$, we can assume by an affine transformation in the plane $\{z=0\}$ that $0\le a \le1$. See Figure~\ref{fig:exception_casef}.

\begin{figure}[h]
\centerline{\includegraphics[scale=1]{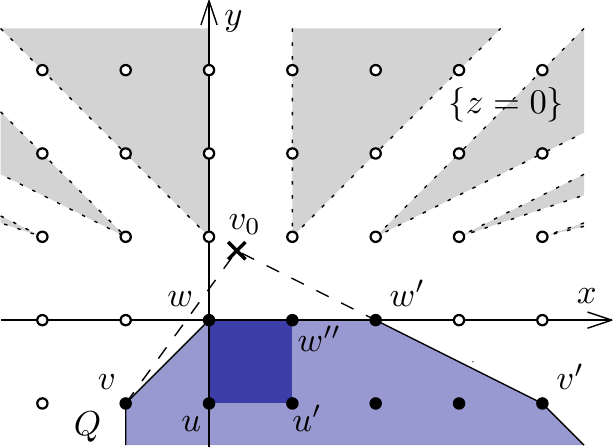}}
\caption{The setting of case (f) in the proof of Theorem~\ref{thm:exception}. Black dots are lattice points in $Q \subsetneq Q_0$, white dots represent other lattice points. The cross represents the intersection $v_0$ of the edge $v_1v_2$ with the plane $\{z=0\}$. The white area in $\{z=0, y>0\}$ is the region where $v_0$ can lie so that $\conv\{w,w'',v_0\}$ has no more lattice points other than $w$ and $w''$.}
\label{fig:exception_casef}
\end{figure}

Let $v$ and $v'$ be the vertices of $Q_0=\conv(Q^{ww'}\cup \{v_0\})$ adjacent to $v_0$ on the sides of $w$ and $w'$, respectively. 
Notice that both $v$ and $v'$ are also vertices of $Q$. 
The wedge formed by the rays from $v_0$ to $v$ and $v'$ contains the points $u=(0,-1,0)$ and $u'=(1,-1,0)$. Since the $y$ coordinate of $v$ and $v'$ is clearly $\le -1$ (or otherwise $w$ or $w'$ would not be vertices of $Q$), those two points are actually in $\conv\{w,w',v,v'\} \subseteq Q$. That is impossible because then $Q$ contains the unimodular parallelogram $ww''u'u$.

\item[(g)] \emph{$Q$ is (isomorphic to) the triangle $\conv\{(-1,-1), (1,0), (0,1)\}$.}\\
The full list of quasi-minimal polygons is worked out in Lemma~\ref{lemma:quasi-minimal-dim2} above and is shown in Figure~\ref{fig:minimal-dim2}.
Configurations 1, 3 and 7 are excluded by Lemma~\ref{lemma:exception-height}\eqref{lemma:exception-height-a}. Let us see how to exclude 2, 4 and 5.

Observe that in the three of them, $Q$ together with any of $v_1$ or $v_2$ is in the conditions of Lemma~\ref{lemma:exception-height}\eqref{lemma:exception-height-b}, which means that $v_i$ is of the form $(a_i,b_i, \pm 2)$ with $a_i \equiv 1\equiv b_i \mod 2$, for both $i=1,2$. $v_1$ and $v_2$ cannot be both on the same side of $Q$, for then their mid-point is another lattice point in $P$ apart from those of $Q$, $v_1$ and $v_2$. So, $v_1$ and $v_2$ are on opposite sides and, in particular, their mid-point has to be one of the lattice points in $Q$. 

Consider first configuration 2 and without loss of generality assume that $Q=\conv\{(-k,0,0),(0,1,0),(0,-1,0)\}$ and $v_1=(1,1,2)$. Its non-essential vertex $(-k,0,0)$ must not be a vertex in $P$, by part (e). Hence, that non-essential vertex is the mid-point of $v_1$ and $v_2$, which means $v_2=(-2k-1,-1,-2)$. But then let $v_3=(0,1,0)$ and let us show that the polytope $P^{v_3}$ has width at least two. A linear functional $f$ giving width one to $P^{v_3}$ must be constant in the collinearities $\{ (0,0,0), (-1,0,0), (-2,0,0)\}$ and $\{ (1,1,2), (-k,0,0),(-2k-1,-1,-2)\}$ which means that $f(x,y,z)= 2y-z$, which is primitive. But then $f(0,0,0)=0$ and $f(0,-1,0) =-2$, which means $P^{v_3}$ has width two. Then $P$ is obtained merging $P^{v_3}$ and $P^{v_i}$, for any choice of $i$, since the polytope $P^{v_3v_i}$ is $3$-dimensional.

In configurations 4 and 5 the argument is the same, except now the mid-point of $v_1$ and $v_2$ can be any of the five lattice points in $Q$. But in all cases there is a vertex $v_3\in Q$ of $P$ such that $P^{v_3}$ has width at least two and $P$ is merged from $P^{v_3}$ and $P^{v_i}$. Details are left to the reader.
\end{enumerate}

Once we know that $Q=\conv\{(1,0,0),(0,1,0),(-1,-1,0)\}$ (modulo isomorphism), then $P$ has six lattice points, so we can rely in the classification of lattice $3$-polytopes of size six and width larger than one, contained in~\cite{6points}. More precisely, Section 6 of that paper studies in detail the configurations consisting of our $Q$ plus another two lattice points. The only ones in which these two lattice points are at lattice distance larger than one from $Q$ (see Sections 6.1.3 and 6.2.2 of~\cite{6points}) are:

\[\left( \begin{array}{cccccc}
0  & 1 & 0 &-1 & 1 & -1\\ 
0  & 0 & 1 &-1 & 2 & -2\\ 
0  & 0 & 0 & 0 & 3 & -3\end{array} \right) \quad
\left( \begin{array}{cccccc}
0  & 1 & 0 &-1 & 1 & -1\\ 
0  & 0 & 1 &-1 & 2 &  1\\ 
0  & 0 & 0 & 0 & 3 & -3\end{array} \right)\]

The first one is the configuration in part (3).
The second one is merged from $P^{v_1}$ and $P^{v_3}$, for $v_3=(1,0,0)$, since $P^{v_3}$ has width larger than one and the polytope 
\[
P^{v_1v_3}=\conv\{(0,0,0),(0,1,0),(-1,-1,0),(-1,1,-3)\}
\]
is $3$-dimensional. 
\qed\end{proof}


\section{The classification of spiked $3$-polytopes}
\label{sec:spiked}

The definition of \emph{spiked} translates in dimension $3$ to Theorem~\ref{thm:projection_shape} below. In its proof the following easy fact is used several times.

\begin{lemma}
\label{lemma:triangles}
Let $T:=\conv\{v_1,v_2,v_3\}$ be a lattice triangle, with $v_3$ at lattice distance more than one from the segment $v_1v_2$. 
Then, there exists a lattice point $p\in T$ at smaller, but non-zero, lattice distance to $v_1v_2$ than $v_3$.
\end{lemma}
\begin{proof}
Without loss of generality let $T=\conv\{(0,0),(0,k),(a,b)\}$ with $k\ge1$ and $a\ge 2$. Consider the triangle $T':=\conv\{(0,0),(0,1),(a,b)\} \subseteq T$. Since $T'$ is not unimodular, it contains some extra lattice point $(c,d)$. It is clear that $0<c<a$.\qed
\end{proof}

In Theorem~\ref{thm:projection_shape} and Corollary~\ref{coro:spiked-projections} reflexive polygons appear: remember that a lattice $d$-polytope $P\subseteq \R^d$ is called \emph{reflexive} if there exists $p\in \inte(P)\cap\Z^d$ such that, for every facet $F\subsetneq P$, the lattice distance between $p$ and $F$ is one. A lattice polygon is reflexive if and only if it has a unique interior lattice point.

\begin{theorem}
\label{thm:projection_shape}
Let $P$ be a lattice $3$-polytope of size at least seven, spiked with respect to a certain $2$-dimensional configuration $A'\subseteq \Z^2$. Let $P':=\conv(A')$. Then one of the following holds:
\begin{enumerate}
\item $P'$ is the second dilation of a unimodular triangle. \label{item:dilated_unimodular}
\item $P'$ is a reflexive triangle and $A'$ consists of the three vertices of it plus its unique interior lattice point. \label{item:reflexive}
\item $A'=P'\cap\Z^2$ and $P'$ has exactly four lattice points in the boundary and one in the interior (That is, $P'$ is one of the three reflexive polygons with four boundary lattice points). \label{item:fourpoints}
\end{enumerate}
\end{theorem}

\begin{proof}
Let $\pi: \R^3 \to \R^2$ be the lattice projection such that $P$ is spiked with respect to $A'=\pi(A)\subseteq \Z^2$, where $A:=P \cap \Z^3$. Remember that $A'$ has width at least two, each vertex of $A'$ has a unique preimage in $A$, at most one vertex of $A'$ is not essential, and $\pi$ bijects the essential vertices of $A$ to the essential vertices of $A'$.

Since the only lattice polygon of width greater than one and without interior lattice points is the second dilation of a unimodular triangle (case~\eqref{item:dilated_unimodular}), for the rest of the proof we assume that $\conv(A')$ has interior lattice points. Let $p'\in\Z^2$ be an interior point of $\conv(A')$.

We are going to prove that $A' \cap \partial P'$ has three or four points, where $\partial P'$ denotes the boundary of the lattice polygon $P'\subseteq \R^2$, and that $A'$ is, respectively, as in case~\eqref{item:reflexive} or case~\eqref{item:fourpoints}.

\begin{itemize}

\item[\eqref{item:reflexive}] If $A'$ has only three points $v'_1$, $v'_2$ and $v'_3$ in the boundary of $P'$, these three points must be essential vertices, because at least three vertices of $A$ are essential and $\pi$ bijects essential vertices. That is, $\conv(A')$ is a triangle and $A'$ is a minimal configuration. $A'$ must have at least one additional lattice point, since $\pi^{-1}(v'_i)$ has a single point in $A$ and $A$ has size at least seven. By assumption, this point has to be in the interior of $\conv(A')$, and we can assume without loss of generality that $p' \in A'$.
We then claim that:

\begin{itemize}
\item \emph{$p'$ is the only point of $A'$ in the interior of $\conv(A')$. That is, $A'=\{v'_1,v'_2,v'_3,p'\}$}. Indeed, if $A'$ has a second interior lattice point $p''$ then let $v'_i$ and $v'_j$ be vertices of $\conv(A')$ on opposite (open) sides of the line containing $p'$ and $p''$. The contradiction is that $\conv\{p',p'',v'_i,v'_j\}\subseteq \conv(A'\setminus\{v'_k\})$ cannot have width one, since one of $p'$ or $p''$ is in its interior.

\item \emph{$p'$ is the only lattice point in the interior of $\conv(A')$. That is, $\conv(A')$ is reflexive with respect to $p'$}. Suppose $p''$ was a second interior lattice point. $p''$ cannot be in the interior of a triangle $\conv\{v'_i,v'_j,p'\}\subseteq \conv(A'\setminus\{v'_k\})$, since these triangles have width one. Thus, $p''$ lies in the relative interior of a segment $\conv\{v'_i,p'\}$. We now apply Lemma~\ref{lemma:triangles} to the triangle $\conv\{v_i,p_1,p_2\}$, where $v_i, p_1, p_2\in A$ are, respectively, the unique point in $\pi^{-1}(v'_i)$ and two points in $\pi^{-1}(p')$ (the latter exist because $A$ has at least seven points, and only one projects to each $v_i$). Existence of $p''$ implies that $v_i$ is at lattice distance at least two from the segment $\conv\{p_1,p_2\}$, so Lemma~\ref{lemma:triangles} says the triangle $\conv\{v_i,p_1,p_2\}$ contains a lattice point $q$ closer to $\conv\{p_1,p_2\}$ than $v_i$. The point $q':=\pi(q)$ is then in $A'$, in contradiction to the fact that $A'=\{v'_1,v'_2,v'_3,p'\}$.
\end{itemize}

\item[\eqref{item:fourpoints}] If $A'$ has at least four points in the boundary of $P'$, let $v_1', ... , v_3'\in A'$ be essential vertices and let $v'_4\in A'$ be another boundary lattice point, which may or may not be a vertex. We assume $v_1', ... , v_4'$ to be cyclically ordered along the boundary. Then:

\begin{itemize}
\item \emph{$v'_1,\dots,v'_4$ are the only lattice points in the boundary of $\conv(A')$, and any other lattice point of $\conv(A')$, in particular $p'$, lies in the relative interior of the segment $v'_2v'_4$}:

Observe that the segment $v'_2v'_4$ decomposes $\conv(A')$ as the union of two polygons $P'_1$ and $P'_3$ contained respectively in $\conv(A'\setminus \{v'_1\})$ and $\conv(A'\setminus \{v'_3\})$, with the point $p'$ lying either in the segment $v'_2v'_4$ or in the interior of one of the two subpolygons (remember that $p'$ is a lattice point in the interior of $\conv(A')$).
Since $\conv(A'\setminus \{v'_1\})$ and $\conv(A'\setminus \{v'_3\})$ have width one, the latter is impossible and $p'$ lies in the relative interior of $v'_2v'_4$. This in turn implies that $v'_1$ and $v'_3$ must be at lattice distance one from the segment.
We also claim that $v'_1$ and $v'_3$ are the only lattice points of $\conv(A')$ outside the segment $v'_2v'_4$. 
If not, let $v'$ be an additional one, say on the side of $v'_1$. Then $\conv(A'\setminus\{v'_1\})$ cannot have width one since it contains three collinear lattice points ($v'_2$, $v'_4$ and $p'$) plus points $v'$ and $v'_3$ on opposite sides of the line containing them.

So, $v'_1$, $v'_2$, $v_3'$ and $v_4'$ are the only boundary lattice points in $\conv(A')$, and $v'_4$ is either a vertex (in which case $\conv(A')$ is a quadrilateral) or it lies in the segment $v'_1v'_3$ (and $\conv(A')$ is a triangle).

\item \emph{$A'$ has some interior lattice point}: Remember that, a priori, $p'$ may not be in $A'$.
Since the preimages of vertices of $A'$ in $A$ consist of a single point and $A$ has at least size seven, if $v'_4$ is a vertex then $A'$ must have some other lattice point. Since the only boundary lattice points are $v'_1,\dots,v'_4$, this has to be an interior point. In case $v_4'$ is not a vertex, if there were no other points in $A'$ then the fiber of $v'_4$ would have at least four lattice points of $A$. Applying Lemma~\ref{lemma:triangles} to the triangle formed by the fiber of $v'_4$ (a segment) and $v'_2$ (a point), and since $p'$ lies in the relative interior of the segment $v'_2v'_4$, we conclude that $A$ has some lattice point projecting to the relative interior of the segment $v_4'v'_2$, a contradiction.

\item \emph{$A'$ cannot have \emph{two} interior lattice points}: if it does, they are both in the segment $v'_2v'_4$, and call $q'$ the closest to $v'_4$. Then $q'$ is in the interior of $ \conv(A'\setminus \{v'_2\})$, which is a contradiction since $v_2'$ is an essential vertex of $A'$. 

That is, we can assume that $p'$ is the only point of $A'$ in the interior of $\conv(A')$. 

\item \emph{$p'$ is the only lattice point in the interior of $\conv(A')$}: For this, we only need to check that $p'v'_2$ and $p'v'_4$ are primitive (a segment is \emph{primitive} if its only lattice points are the endpoints).
If the fiber of $p'$ in $A$ has at least two points, then Lemma~\ref{lemma:triangles} applied to the triangle formed by these two points plus $v_2$ (resp.~$v_4$) implies that $p'v_2'$ (resp.~ $v_4'p'$) is primitive: otherwise, $A$ must have lattice points projecting to the relative interior of $p'v'_2$ (resp. $p'v'_4$).
If the fiber of $p'$ in $A$ has a single point then the fiber of $v'_4$ must have at least three and the same argument shows that $v'_4p'$ is primitive, but we need an extra argument for $p'v'_2$.

So, suppose that $p'$ has a single point $p$ in its fiber, which implies $v'_4$ has at least three. 
Call $v^+_4$, $v^0_4$ and $v^-_4$ three consecutive lattice points of $A$ in the fiber of $v'_4$, and call $p^+:= p + v^+_4-v^0_4$ and $p^-:= p + v^-_4-v^0_4$. That is, $p^+$, $p$ and $p^-$ are consecutive points projecting to $p'$, in the same order as $v^+_4$, $v^0_4$ and $v^-_4$. Let $v_2$ be the unique point of $A$ in the fiber of $v'_2$. See Figure~\ref{fig:diff_triangle}.
Since the triangle $v^+_4v^-_4v_2\subseteq \conv(A)$ does not contain $p^+,p^- \not\in A$, $v_2$ must lie in the ray $r$ from $v^0_4$ through $p$. Then the lattice points in the segment $pv_2$ are all in $A$, but no such point can arise other than $p$ and $v_2$ because it would project to a point of $A'$ in the relative interior of the segment $p'v'_2$, which does not exist. Hence, $pv_2$ is primitive. Now, since $pv_4^0$ is a primitive segment projecting to $p'v'_4$, which is also primitive, and since $pv_4^0$ and $pv_2$ are parallel, this implies that also $p'v'_2$ is primitive.\qed
\begin{figure}[ht]
\centerline{
\includegraphics[scale=.9]{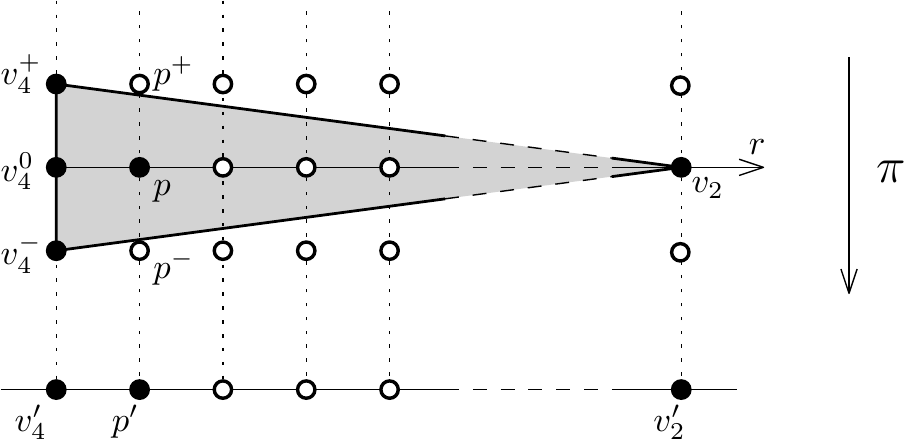}}
\caption{The situation in the final part of the proof of Theorem~\ref{thm:projection_shape}. Black dots represent lattice points of $A$ and $A'$. White dots represent other lattice points. The gray triangle is $\conv\{v^+_4,v^-_4,v_2\}$.}
\label{fig:diff_triangle}
\end{figure}
\end{itemize}
\end{itemize}
\end{proof}

\begin{corollary}
\label{coro:spiked-projections}
A spiked $3$-polytope is spiked with respect to one of the ten quasi-minimal configurations $A'_1,\dots,A'_{10}$ of Figure~\ref{fig:2dim_projections}.
\begin{figure}[ht]
\centerline{\includegraphics[scale=0.8]{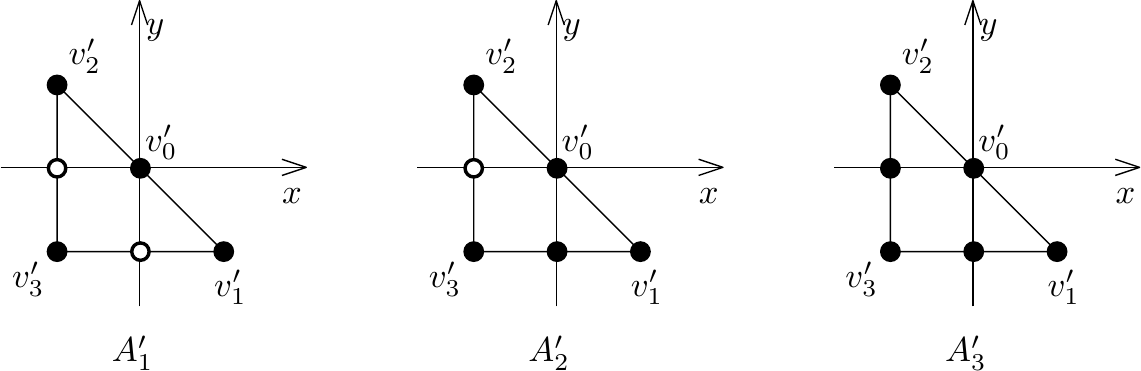}}
\centerline{\includegraphics[scale=0.8]{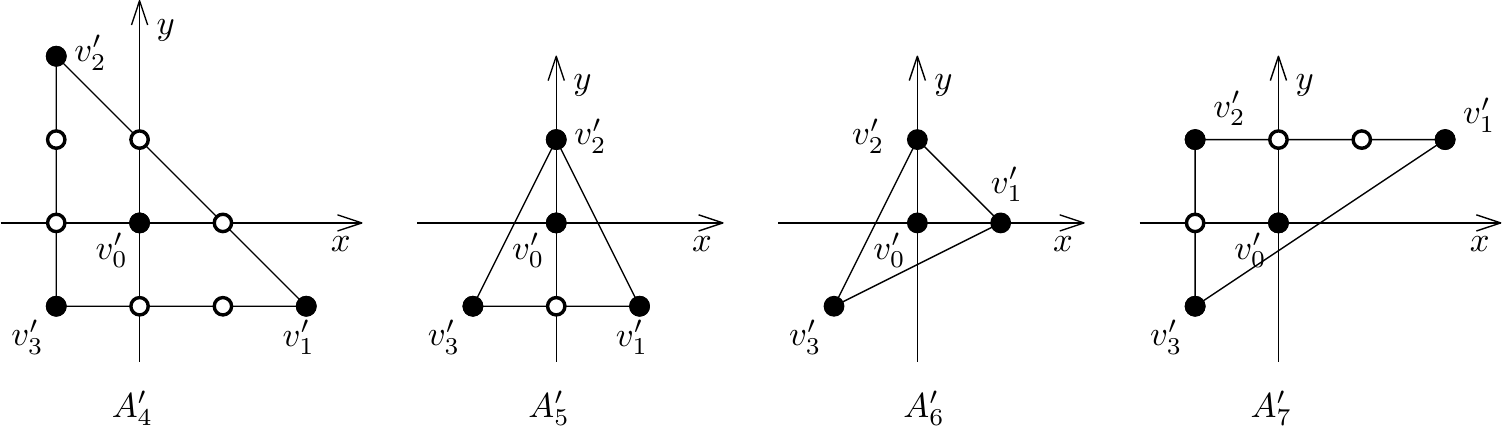}}
\centerline{\includegraphics[scale=0.8]{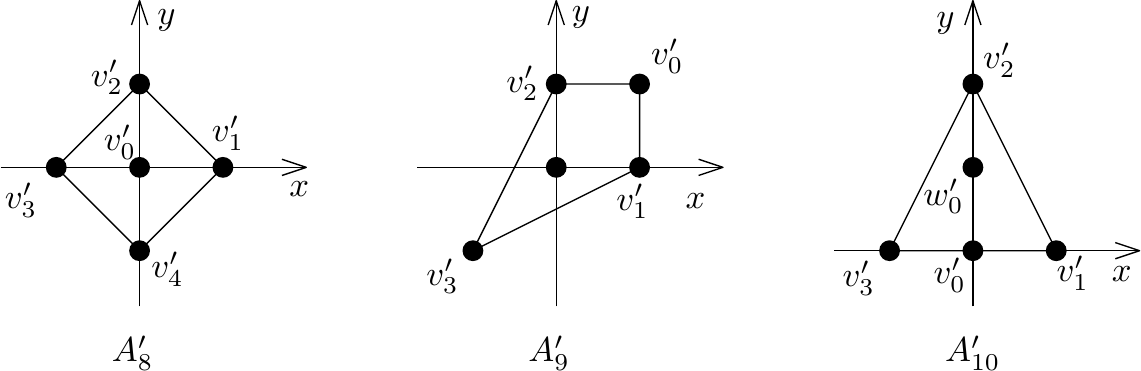}}
\caption{The ten quasi-minimal configurations that can arise as the projection of a spiked $3$-polytope. The configuration $A'_i$ consists in each case of the black dots. White dots are other lattice points in $\conv(A'_i)$. Labels $v'_i$ of certain lattice points are there for reference in the proof of Theorems~\ref{theorem:spiked-minimals} and \ref{theorem:spiked-quasiminimals}.}
\label{fig:2dim_projections}
\end{figure}
\end{corollary}

\begin{proof}
Let $P$ be a spiked $3$-polytope and let $A':=\pi(P\cap \Z^3)\subseteq\Z^2$ be the quasi-minimal configuration with respect to which $P$ is spiked, for $\pi:\R^3 \to \R^2$ a lattice projection. 
Without loss of generality assume that $\pi$ is the projection that forgets the third coordinate.

Let us look at the three cases allowed by Theorem~\ref{thm:projection_shape} for $P':=\conv(A')$:

\begin{enumerate}
\item[\eqref{item:dilated_unimodular}] If $P'$ is the second dilation of the unimodular triangle, all three vertices of $P'$ are in $A'$ but the mid-points of edges may or may not be in $A'$. The statement simply says that at least one of them is in $A'$. (This is the only case missing from the top row of Figure~\ref{fig:2dim_projections}). 
This must be so because $A$ has at least seven points and only three of them project to the vertices of $A'$ (the three vertices are essential).

\item[\eqref{item:reflexive}] Suppose $P'$ is a reflexive triangle and the unique points of $A'$ are the three vertices and the unique interior point of $P'$. There are five reflexive triangles; the four in the middle row of Figure~\ref{fig:2dim_projections} plus the following one (the black dots in the figure are lattice points in $A'$, and the white dots are lattice points in $P'\setminus A'$):

\centerline{
\includegraphics[scale=0.7]{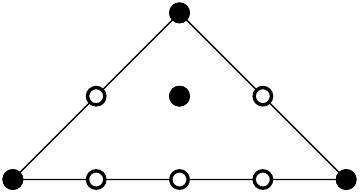}}
\label{fig:reflexive_exception}
But, if $P'$ was this triangle then two of the vertices of $P$ should be at heights of the same parity and their mid-point would be a lattice point in $A$. In particular, one of the midpoints of edges of $P'$ would by in $A'$, a contradiction.

\item[\eqref{item:fourpoints}] In case \eqref{item:fourpoints} of the theorem $P'$ is a reflexive polygon with four boundary points and $A'=P'\cap\Z^2$. The last row of Figure~\ref{fig:2dim_projections} shows all three possibilities.\qed
\end{enumerate}
\end{proof}

\begin{theorem}[Classification of spiked minimal $3$-polytopes]
\label{theorem:spiked-minimals}
Let $P$ be a spiked minimal lattice $3$-polytope of size at least seven. Then $P$ is equivalent to
\[
\conv\{(1,0,0),(0,1,0),(-1,0,-a),(0,-1,2k+b)\}
\]
for some $(a,b) \in \{(0,0),(0,1),(1,1)\}$ and an integer $k\ge 2$. It has size $k+5$.
\end{theorem}

\begin{proof}
By Corollary~\ref{coro:spiked-projections}, $P$ is spiked with respect to one of the configurations $A'_i$ in Figure~\ref{fig:2dim_projections}. Since $P$ has at least four essential vertices, so does $A'_i$, which leaves only the possibility $A'_8$.
We use the coordinates and labels from Figure~\ref{fig:2dim_projections}, and assume that the projection is the one that forgets the $z$ coordinate. 

By definition of spiked the lattice points in $P$ are its four essential vertices $v_i$ that project to each $v_i'$, plus the lattice points projecting to $(0,0)$, none of which are vertices. We assume that there are $k+1$ such points and they form the segment $S_k:=\{(0,0,0),(0,0,1),\dots, (0,0,k)\}$ for some $k$ (with $k\ge 2$ or otherwise $P$ has size less than seven). Since the triangle $v'_0v'_1v'_2$ is unimodular, we can arbitrarily change the heights of $v_1$ and $v_2$ keeping the choices so far, so we assume $v_1=(1,0,0)$ and $v_2=(0,1,0)$. 
Let $\ell$ be the vertical line $\{x=y=0\}$. 
In order for the fiber of $v'_0$ in $P\cap \Z^3$ to equal $S_k$, one of the segments $v_1v_3$ and $v_2v_4$ must cut $\ell$ at height in $(-1,0]$, and the other at height in $[k,k+1)$. That is, without loss of generality, $v_3=(-1,0,-a)$ and $v_4=(0,-1,2k+b)$ for $a,b\in\{0,1\}$. Furthermore, by the affine symmetry $(x,y,z) \to (y,x,-kx-ky-z+k)$ (which exchanges the values of $a$ and $b$), we can assume that $a\le b$. \qed
\end{proof}

\begin{theorem}[Classification of spiked quasi-minimal $3$-polytopes]
\label{theorem:spiked-quasiminimals}
Let $P$ be a spiked quasi-minimal but not minimal lattice $3$-polytope of size at least seven. Then $P$ is equivalent to one of the following. In all cases $k\ge 2$ is an integer and the point in boldface is the non essential vertex of $P$.
\begin{enumerate}
\item[(1)] $\conv\{(1,-1,-1),(-1,1,1),(-1,-1,0),{\bf (0,0,k)}\}$, of size $k+4$. 
\smallskip

\item[(2)] $ \conv\{(1,-1,0),(-1,1,-1),(-1,-1,0),{\bf (0,0,k)}\}$, of size $k+5$. 
\smallskip

\item[(3)] $ \conv\{(1,-1,0),(-1,1,0),(-1,-1,0),{\bf (0,0,k)}\}$, of size $k+6$. 
\smallskip

\item[(4)] $\conv\{(2,-1,-1),(-1,2,1),(-1,-1,0),{\bf (0,0,k)}\}$, of size $k+4$. 
\smallskip

\item[(5)] $ \conv\{(1,-1,-1),(0,1,a),(-1,-1,0),{\bf (0,0,k)}\}, a\in \{-1,0\}$, of size $k+4$. 
\smallskip

\item[(6)] $ \conv\{(1,0,0),(0,1,a),(-1,-1,0),{\bf (0,0,k)}\}, \ a \in \{-2,-1,0\}$, of size $k+4$.
\smallskip

\item[(7)] $ \conv\{(2,1,0),(-1,1,a),(-1,-1,0),{\bf (0,0,k)}\}, \ a \in \{-5,-1\}$, of size $k+4$.
\smallskip

\item[(8)] $ \conv\{(1,0,0),(0,1,0),(-1,0,a),(0,-1,b),{\bf (0,0,k)}\}, \  a \in \{-1,0\}, a \le b < 2k$, of size $k+5$. 
\smallskip

\item[(9)] $ \conv\{(1,0,0),(0,1,0),(-1,-1,a),{\bf (1,1,2k-a+b)}\}, \ a \in \{-2,-1,0\}, b \in \{0,1\}$, of size $k +5$.
\smallskip

\item[(10a)] $\conv\{(1,0,a),(0,2,b),(-1,0,0),{\bf (0,0,k)}\}, \ a, b \in \{-1,0\}$, of size $\lfloor (3k+b)/2 \rfloor +5$.
\smallskip

\item[(10b)] $ \conv\{(1,0,0),(0,2,a),(-1,0,0),{\bf (0,1,k)}\}, \ a \in \{-1,0\}$, of size $k+5$. 
\end{enumerate}
\end{theorem}

\begin{proof} 
By Corollary~\ref{coro:spiked-projections}, $P$ is spiked with respect to one of the ten configurations $A'_i$ in Figure~\ref{fig:2dim_projections}. 
This will correspond to the ten cases in the statement, except that case (10) subdivides into two subcases as we show below.
Without loss of generality we can assume that the projection is the one that forgets the $z$ coordinate and we take in $\Z^2$ the system of coordinates of Figure~\ref{fig:2dim_projections}. 

Let us first concentrate on cases (1) to (7), in which $A'_i$ is a minimal configuration and $\conv(A_i')$ a triangle (with vertices $v'_1$, $v'_2$ and $v'_3$ as labeled in the figure). 
This implies that $P$ is a tetrahedron with three essential vertices $v_1$, $v_2$, $v_3$ projecting to the three vertices of $\conv(A'_i)$, plus a non-essential vertex $v_0$ projecting to a non-vertex lattice point $v'_0$ of $A'_i$. In all except $A'_2$ and $A'_3$ there is only one choice for $v'_0$. In cases $A'_2$ and $A'_3$ there are several possibilities for $v'_0$ but they are equivalent to one another. This allows as to assume $v'_0$ is as shown in the figure in all cases. Hence, to finish the proof for these seven cases we only need to derive the possible third coordinates (the heights) for the four vertices $v_i$, in each case. We denote these heights $h_0$, $h_1$, $h_2$ and $h_3$, and let us check that without loss of generality they are as in the statement:
\begin{itemize}
\item Let $k+1$ be the number of lattice points in the fiber of $v'_0$ in $P\cap \Z^3$. We take without loss of generality $v_0=(0,0,k)$ (that is, $h_0=k$), so that $(0,0,0)$ is the bottom-most point in the fiber. Observe that $v'_0$ is the only point of $A'_i$ whose fiber in $P\cap \Z^3$ has more than a single point. For all except $A'_2$ and $A'_3$ this is obvious, since $v'_0$ is the only non-vertex. For $A'_2$ and $A'_3$, the fibers of $(0,-1)$ (in both) and $(-1,0)$ (in $A'_3$) must be single points or otherwise they produce additional vertices in $P$, which do not exist.\smallskip

\item Since the segment $v'_0v'_3$ is primitive, there is no loss of generality in taking height zero for $v_3$. That is to say, $h_3=0$ and $v_3=(-1,-1,0)$, in all seven cases.\smallskip

\item Let $d_i\ge 1$ be the lattice distance from $v'_1$ to the line $\{x=y\}$ spanned by $v'_0v'_3$. (That is, $d_i=2,2,2,3,2,1$ and $1$, respectively, in cases (1) to (7)). 
Since the unimodular transformation $(x,y,z)\to (x,y,z\pm(x-y))$ fixes the plane containing all the choices so far (points projecting to $v'_0$ and $v'_3$) and changes the height of $v_1$ by $d_i$ units, without loss of generality we choose the height of $v_1$ to be in $\{0,-1,\dots, -d_i +1\}$.
Moreover, this height must be even in cases (2) and (3), in order for the midpoint of $v_1v_3$ to be a lattice point, and it must be relatively prime to $d_i$ in all other cases, in order for the segment $v_1v_3$ to be primitive. Summing up, the height $h_1$ of $v_1$ equals: $0$ in (2), (3), (6), and (7); $-1$ in (1) and (5); and $-1$ or $-2$ in (4). \smallskip

\item In order to study $h_2$, let $h$ be the height at which the triangle $v_1v_2v_3$ intersects the vertical line projecting to $v'_0$. Since our choice is that $(0,0,0)$ is the bottom-most lattice point in the fiber of $v'_0$, we must have $h\in (-1,0]$. This in turn implies a bounded interval for the height $h_2$ in each case, namely:

\begin{itemize}
\item[(1)] $h_2 \in(-1,1]$.
\hfil 
\ \ \ \quad(5) $h_2 \in (-3/2,1/2]$.

\item[(2)] $h_2 \in(-2,0]$.
\hfil 
(6) $h_2 \in (-3,0]$.

\item[(3)] $h_2 \in (-2,0]$.
\hfil 
(7) $h_2 \in (-6,0]$.

\item[(4)] $h_2 \in (-2, 1]$ if $h_1=-1$ and $h_2 \in (-1, 2]$ if $h_1=-2$.\smallskip
\end{itemize}

\item This already gives a finite list of possibilities for all heights, but there are the following additional considerations:

\begin{itemize}
\item In (3), $h_2$ must be even in order for the midpoint of $v_2v_3$ to be integer.
\item In (1), (2) and (7), $h_2$ must be odd for the segment $v_2v_3$ to be primitive.
\item In (4), $h_2 \ne 0 \pmod 3$ for the segment $v_2v_3$ to be primitive.
\item In (4) and (7), $h_2 \ne h_1 \pmod 3$ for the segment $v_1v_2$ to be primitive.
\end{itemize}
Together with the intervals stated above, this fixes $h_2$ to be $1$, $-1$, and $0$, in cases (1), (2) and (3), respectively. In case (4) we have two possibilities for $(h_1,h_2)$, namely $(-1,1)$ and $(-2,2)$, but they produce equivalent configurations via the transformation $(x,y,z)\mapsto(y,x,x-y+z)$), so we take the first one.
In cases (5), (6) and (7) we have $h_2\in \{-1,0\}$, $h_2\in \{-2,-1,0\}$, and $h_2\in \{-5,-1\}$, respectively. This finishes the proofs of these seven cases.

\end{itemize}

We now look at the three remaining cases, $A'_8$, $A'_9$, and $A'_{10}$. As before, we will denote by $h_i$ the height of the vertex $v_i$ of $P$ projecting to a point $v_i'\in A_j'$, with $i\in\{0,\dots,3\}$ and $j\in \{8,9,10\}$. The ideas are essentially the same, with slight modifications:
\begin{itemize}

\item[(8)] $A'_{8}$ is minimal, hence in this case $P$ has four essential vertices $v_1,v_2,v_3,v_4$ projecting to the four vertices of $A'_8$ plus a fifth non-essential vertex $v_0$ projecting to $v'_0=(0,0)$. 
Again, we choose $v_0=(0,0,k)$ where $k+1$ is the number of lattice points in the fiber of $v'_0$, and as in Theorem~\ref{theorem:spiked-minimals}, we can take without loss of generality, the heights of $v_1$ and $v_2$ to be zero. By symmetry, we can also assume $h_3\le h_4$ which implies, in order for the bottom-most point in the fiber of $v'_0$ to be $(0,0,0)$, that $h_3\in \{-1,0\}$. 
Finally, in order for $v_0$ to be above the segment $v_2v_4$ we need $h_4 < 2k$.\smallskip

\item[(9)] $A'_{9}$ is not minimal, so the vertices of $P$ biject to those of $A'_{9}$ and all other lattice points in $P$ project to the unique non-vertex point $p_0'=(0,0)$. We let the fiber of $p'_0$ consist of $(0,0,0),\dots, (0,0,k)$, as in previous cases. Since the triangle $p'_0v'_1v'_2$ is unimodular we can choose the heights of $v_1$ and $v_2$ to be zero. Also, we make the choice that the triangle $v_1v_2v_3$ lies below (perhaps not strictly) $(0,0,0)$ and the segment $v_0v_3$ lies above. The opposite choice would lead to equivalent configurations.

Then, for the triangle $v_1v_2v_3$ to cut the line $\{x=y=0\}$ at height in $(-1,0]$ we need $h_3\in (-3,0]$. And for the segment $v_0v_3$ to cut that line at height in $[k,k+1)$ we need $(h_0+h_3)/2 \in [k, k+1)$. Hence $h_0\in\{2k-h_3,2k-h_3+1\}$ and $k=\lfloor(h_0+h_3)/2\rfloor$.\smallskip

\item[(10)] $A'_{10}$ is minimal, which implies $P$ to have three essential vertices $v_1,v_2,v_3$ projecting to the three vertices of $A'_{10}$, plus a fourth non-essential vertex projecting to one of the other two lattice points, $v'_0=(0,0)$ and $w'_0=(0,1)$. We consider the two cases separately:\smallskip
\begin{itemize}
\item[(10a)] If the non-essential vertex projects to $v'_0$, call it $v_0$. By the same arguments as used for configurations $(1)$ to $(7)$, we can assume that $v_0=(0,0,k)$, $(0,0,0)$ is the bottom-most point in the fiber of $v_0'$, $h_3=0$ and $h_1 \in \{-1,0\}$. Once these are fixed, unimodular transformations can change the height of $v_2$ by arbitrary even numbers, so we can take the height of $v_2$ in $\{0,-1\}$ as well. 
Observe that the fiber of point $w_0'$ in $P$ is the segment going from $(0,1,h')$, with $h'=(h_1+2 h_2)/4\in (-1,0]$, to $(0,1,(h_2 + k)/2)$. It then contains the $\lfloor (h_2+k)/2 \rfloor+1$ lattice points from $(0,1,0)$ to $(0,1,\lfloor (h_2+k)/2 \rfloor)$.\smallskip

\item[(10b)] If the non-essential vertex projects to $w'_0$, call it $w_0$. By the same arguments as before, we can assume that $w_0=(0,1,k)$, $(0,1,0)$ is the bottom-most point in the fiber of $w_0'$, $h_3=0$ and $h_1 \in \{-1,0\}$. In this case, the fiber of $v_0'$ consists of a single point, the middle point of segment $v_1v_3$. In order for this point to be a lattice point, $h_1$ has to be even, hence $h_1=0$. Then, in order for the triangle $v_1v_2v_3$ to cut the fiber of $w_0'$ at a height in $(-1,0]$, we need $h_2\in\{-1,0\}$.\medskip

\end{itemize}
\end{itemize}
In all cases $k$ can be assumed at least two: In case $A'_3$ because otherwise $P$ has width one with respect to the vertical direction. In all other cases because otherwise $P$ has size at most $6$. \qed
\end{proof}

\begin{remark}
\label{remark:spiked}
\rm
For $k\ge 3$ all the polytopes described in Theorem~\ref{theorem:spiked-quasiminimals} are spiked, quasi-minimal, not minimal, and have size $\ge 7$. 
They are also non-isomorphic to one another. Each polytope has at least four lattice points in one of the fibers. Since there is no other direction these polytopes can have four aligned lattice points, one single polytope cannot be spiked with respect to different configurations $A'_i$. And among the polytopes that are spiked with respect to one specific $A'_i$, the choices of coordinates for their lattice points have been made so that no two polytopes are equivalent (see details in the proof).\medskip

But for $k=2$ the following happens: 
\begin{itemize}
\item Cases (1), (4), (5), (6), and (7) produce size six.
\item In some cases (sometimes depending also on the values of $a$ and $b$) the vertex that should be non-essential (the vertex $v_0$ or $w_0$ in the proof, projecting to $v'_0$ or $w'_0$ in Figure~\ref{fig:2dim_projections}) turns out to be essential. In this case the polytope obtained is minimal, and it is not spiked with respect to that projection: it no longer bijects essential vertices to essential vertices.
\end{itemize}

This means that for each size $n\ge 9$ there are \emph{exactly} the following non-equivalent spiked $3$-polytopes: $3$ spiked minimal tetrahedra; $23$ (if $n=0 \pmod 3$) or $21$ (if $n \ne 0 \pmod 3$) spiked quasi-minimal, not minimal tetrahedra; and $4n-19$ spiked quasi-minimal, not minimal $3$-polytopes with $5$ vertices. For $n=7$ and $8$ the global counts are decreased by two. See exact numbers in Table~\ref{table-minimal}.
\end{remark}

\begin{remark}
\label{rem:not_boxed}
\rm
Observe as well that no quasi-minimal $3$-polytope of size at least seven can be both spiked and boxed. 
Indeed, with $k\ge 2$ in Theorems~\ref{theorem:spiked-minimals} and~\ref{theorem:spiked-quasiminimals} the only way a polytope $P$ can be boxed and spiked is if there exists an essential vertex $v$ such that $P^v$ has width one with respect to a functional that is not constant on the fibers of the projection (see Remark~\ref{rm:clarify_boxed_vs_spiked}). This implies that each fiber can contain at most two lattice points of $P^v$. Remember that $k+1$ equals the maximum number of lattice points of $P$ contained in the same fiber. 
Since $v$ is an essential vertex, by definition of spiked this lattice point is alone in its fiber, and the maximum number of lattice points in a fiber of $P$ is still $k+1\le 2$, which is a contradiction.
\end{remark}


\section{The classification of boxed $3$-polytopes}
\label{sec:boxed}

Let $P$ be a boxed $3$-polytope of size at least seven. That is to say, there are three integer primitive affine functionals $f_1,f_2,f_3:\R^3\to \R$ such that the lattice points in $P$ are:
\begin{itemize}
\item Some or all of the vertices of the rational parallelepiped $Q:=\bigcap_{i=1}^3 f_i^{-1} [0,1]$. 
\item Three additional points $v_1$, $v_2$, $v_3$ (essential vertices of $P$) with $f_i(v_j)\not\in\{0,1\}$ if, and only if, $i=j$.
\end{itemize}
Without loss of generality we assume the origin to be a vertex of $Q$, so that the $f_i$'s can be taken integer primitive \emph{linear} functionals.

For each $i \in \{1, 2, 3\}$, let
\[
C_i^+:=\big(\bigcap_{j\neq i} f_j^{-1} [0,1]\big) \cap f_i^{-1} (1,\infty)
\]
and
\[
C_i^-:=\big(\bigcap_{j\neq i} f_j^{-1} [0,1]\big) \cap f_i^{-1} (-\infty, 0),
\]
and let $C_i:=C_i^+ \cup C_{i}^-$. We call the $C_i$'s \emph{chimneys} of $Q$ and refer to $C_i^+$ and $C_i^-$ as \emph{half-chimneys}. With this notation, $v_i\in C_i$ for each $i$.
See Figure~\ref{fig:chimneys}.

\begin{figure}[ht]
\includegraphics[scale=1.05]{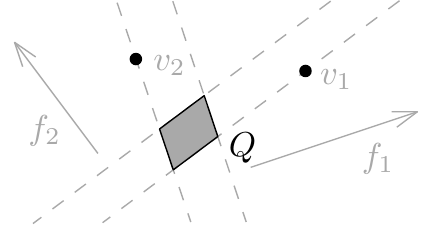}
\includegraphics[scale=1.05]{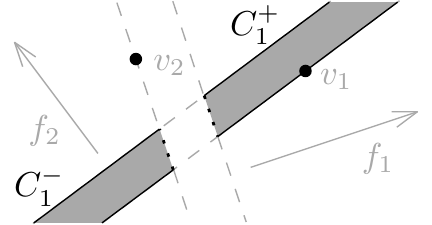}
\caption{The chimneys of a boxed polygon.}
\label{fig:chimneys}
\end{figure}

In order to classify boxed $3$-polytopes, in this section we do the following: in Section~\ref{sec:parallelepipedQ} we look at the possibilities for $Q$ and prove that all boxed $3$-polytopes of size at least seven are boxed with respect to either the unit cube or one specific rational parallelepiped $Q_0$. Once we know that $Q$ is one of these two parallelepipeds, in Section~\ref{subsec:v_i} we use their coordinates to bound the possibilities for vertices $v_i$, which a priori are infinitely many. Finally, in Section~\ref{subsec:computer-routines} we explain how we use the theoretical results to actually implement computer algorithms that enumerate all boxed $3$-polytopes.

\subsection{Possibilities for the parallelepiped}
\label{sec:parallelepipedQ}

The Euclidean volume of the parallelepiped $Q$ equals the inverse of the determinant of $(f_1,f_2,f_3)$, which is an integer. In particular, the volume of $Q$ is exactly one if and only if $Q\cong [0,1]^3$, and is at most $1/2$ otherwise. The following lemma shows that, if we restrict ourselves to boxed $3$-polytopes of size at least seven, there is only one other possibility for $Q$.

\begin{lemma}
\label{lemma:cubes}
Let $P$ be a boxed $3$-polytope with size at least seven and suppose $P$ is not boxed with respect to a parallelepiped unimodularly equivalent to the standard cube. Then, modulo unimodular equivalence, we can assume that $f_1=y+z$, $f_2=x+z$, $f_3=x+y$ so that $P$ is boxed with respect to the parallelepiped
\begin{eqnarray*}
Q_0&:=&\conv\bigg\{(0,0,0),
\left(-\frac{1}{2},\frac{1}{2},\frac{1}{2}\right),\left(\frac{1}{2},-\frac{1}{2},\frac{1}{2}\right),\left(\frac{1}{2},\frac{1}{2},-\frac{1}{2}\right),
 \\ 
&&  (1,0,0), (0,1,0), (0,0,1), \left(\frac{1}{2},\frac{1}{2},\frac{1}{2}\right) \bigg\}.
\end{eqnarray*}
\end{lemma}

\begin{figure}[ht]
\centerline{
\includegraphics[scale=1]{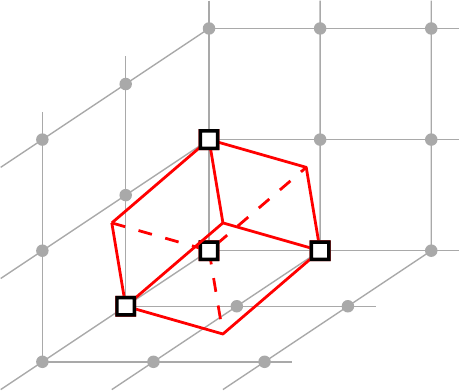}
}
\caption{The parallelepiped $Q_0$. White squares are the lattice points in it. Gray dots are other lattice points.}
\end{figure}

\begin{proof}
Let $P$ be of size at least seven and boxed with respect to a parallelepiped $Q$ not unimodularly equivalent to the unit cube. As usual, let $P \cap \Z^3= A \cup \{v_1,v_2,v_3\}$ where $A\subseteq Q\cap\Z^3$ has size at least four.

If $T \subseteq A$ consists of four non-coplanar lattice points, the convex hull of them is a lattice tetrahedron whose vertices are vertices of the parallelepiped $Q$. It is easy to see that the only two possibilities are that either $\conv(T)$ has three vertices in a common facet of $Q$, or that no two vertices of $\conv(T)$ share the same edge of $Q$. In each case, the volume of this tetrahedron is one sixth and one third, respectively, of the volume of $Q$. Since the Euclidean volume of $Q$ is at most $1/2$, the Euclidean volume of $\conv(T)$ in the first case is $\le 1/12$, which contradicts the fact that any lattice $3$-polytope has Euclidean volume at least $1/6$. 
That is, the only possibility is that $\conv(T)$ consists of alternating vertices of $Q$ and it has Euclidean volume exactly $1/6$, which implies that $\conv(T)$ is a unimodular tetrahedron. Since all unimodular tetrahedra are equivalent, there is no loss of generality in assuming $T=\{(0,0,0), (1,0,0), (0,1,0), (0,0,1)\}$.
The three primitive linear functionals with values $0$ and $1$ on opposite edges of $\conv(T)$ are $x+y$, $y+z$ and $x+z$, as in the statement. Hence $Q=Q_0$ and $A=T$.

So, for the rest of the proof we assume that $A$ is contained in a plane (in particular, it has exactly four points) and try to get a contradiction. The two possibilities are that the points in $A$ are either the vertices of a facet of $Q$ or the vertices of two opposite parallel edges:

\begin{itemize}

\item If $A$ consists of the four vertices of a facet $F$, then let $H$ be the plane containing $F$, and let $F'$ be the opposite facet and $H'$ the plane containing it. Since $F$ is an empty lattice parallelogram, $F$ is a fundamental parallelogram of the lattice $H\cap \Z^3$. By translation, and since $H'$ is a lattice plane by definition of boxed, $F'$ is a fundamental parallelogram of $H'\cap \Z^3$, a contradiction with the fact that $F'$ contains no lattice points.

\item If $A$ consists of the vertices of two opposite parallel edges, then the four points still form an empty parallelogram. Assume without loss of generality that
\[
A=\conv\{(0,0,0),(1,0,0),(0,1,0),(1,1,0)\} \subseteq Q \subseteq \{z=0\}
\]
where the lines $\{x=z=0\}$ and $\{x=1, z=0\}$ contain opposite edges of $Q$. Since $Q$ has no other lattice points, and the points in $A$ are in opposite parallel edges, one of the chimneys, say $C_1$, contains these two lattice lines and no other. Hence $v_1$ is in $\{x=z=0\}$ or in $\{x=1, z=0\}$, with $y$ coordinate in $\Z\setminus \{0,1\}$.
By symmetry of the conditions so far with respect to the planes $\{x=1/2\}$ and $\{y=1/2\}$, we can assume without loss of generality that $v_1=(0,b_1,0)$, for $b_1\ge 2$. But for any choice of $b_1$ the point $(0,2,0)$ is in $P$, so we must actually have $v_1=(0,2,0)$.

Now, vertices $v_2$ and $v_3$ must satisfy that $\conv(A\cup\{v_i\})$ ($i\in\{2,3\}$) does not have any extra lattice points. Since $A$ contains a unimodular parallelogram, $v_2$ and $v_3$ must be at lattice distance at most one from $A$ (Lemma~\ref{lemma:exception-height}(1)). Moreover, they must be in opposite sides of $A$ or otherwise $P$ has width one. That is, without loss of generality we can assume $v_2=(0,0,1)$, and $v_3=(a,b,-1)$, for some $a,b \in \Z$.

The functional $f_1$ has to be equal to zero on the segment $(0,0,0)(1,0,0)$, and to $1$ in the segment $(0,1,0),(1,1,0)$, so it has the form $f_1(x,y,z)= y+cz$ for some $c\in \Z$. By definition of boxed, we need to have that $v_2,v_3 \in f_1^{-1}(\{0,1\})$, which implies that $f_1(v_2)=f_1(0,0,1)=c\in\{0,1\}$ and $f_1(v_3)=f_1(a,b,-1)=b-c\in\{0,1\}$. In particular, $b\in\{c,c+1\} \subseteq \{0,1,2\}$. 

Let
\[
Q' := \bigcap_{i=1}^3 f'^{-1}_i([0,1]),
\]
for $f'_2=-z$, $f'_3=x$ and $f'_1$ equal to $y$ if $b\in\{0,1\}$ and to $y+z$ if $b=2$. It turns out that $P$ is also boxed with respect to the parallelepiped $Q'$, since $f'_i(P^{v_i})\subseteq [0,1]$ and $f'_i(v_i) \not\in [0,1]$, for all $i$ (observe that $f'_3(v_3)= a \not\in \{0,1\}$ follows because $a\in \{0,1\}$ gives $P$ width one with respect to $x$).
Since $Q'\cong [0,1]^3$, this is a contradiction.
\qed
\end{itemize}
\end{proof}

\subsection{Possibilities for the vertices $v_i$}
\label{subsec:v_i}

A priori, $v_i$ can be any of the (infinitely many) lattice points in the chimney $C_i$. In this section we give bounds on how far $v_i$ can be from $Q$, which reduces the infinite possibilities to finitely many.

For each $i\in\{1,2,3\}$ denote by $r_i$ the (unique) line that contains $v_i$ and an edge of $Q$, and let $s_i:=r_i \cap Q$ be such edge. 
In case $s_i$ contains a lattice point of $P$, bounding the possible positions of $v_i$ is quite straightforward. (We assume $d=3$, but Lemma~\ref{lemma:firstpoint} and Corollary~\ref{coro:firstpoint} are valid in arbitrary dimension):

\begin{lemma}
\label{lemma:firstpoint}
Let $P$ be boxed with respect to a parallelepiped $Q$, and let $v_i$ be one of the three lattice points in $P\setminus Q$. If $s_i$ contains a lattice point of $P$, then there is no lattice point along the line $r_i$ strictly between $s_i$ and $v_i$.
\end{lemma}

\begin{proof}
Let $q$ be a lattice point in $s_i\cap P$. If there was a $p \in r_i \cap \Z^3$ strictly between $s_i$ and $v_i$ then $p\in P$ would neither be in $Q$ nor be a vertex of $P$ (since it lies in the segment from $q$ to $v_i$). This is a contradiction with the definition of boxed.\qed
\end{proof}


\begin{corollary}
\label{coro:firstpoint}
Let $P$ be boxed with respect to a parallelepiped $Q$. If all edges of $Q$ contain lattice points of $P$ then each $v_i$ is the first lattice point in one of the eight rays in the corresponding chimney. 
\end{corollary}

This allows us to fully understand boxed $3$-polytopes with respect to the parallelepiped $Q_0$ of Lemma~\ref{lemma:cubes}. Since $Q_0$ contains only four lattice points and we assume $P$ has size at least seven, we conclude that $P$ has size exactly seven and consists of those four lattice points plus $v_1$, $v_2$ and $v_3$. Moreover, since those four lattice points are alternate vertices of $Q_0$, $P$ contains lattice points in all edges of $Q_0$ and Corollary~\ref{coro:firstpoint} implies:

\begin{corollary}
\label{coro:vertices_Q_0}
Let $P$ be a boxed $3$-polytope of size at least seven and suppose that it is not boxed with respect to a parallelepiped unimodularly equivalent to $[0,1]^3$. Then $P\cap \Z^3 \cong \{(0,0,0),$ $(1,0,0), (0,1,0),(0,0,1),v_1,v_2,v_3\}$ with
\[
v_1\in
  \left\{
    \begin{matrix}
      (-1,\;\;\;1,\;\;\;1)\\(-1,\;\;\;1,\;\;\;2)\\(-1,\;\;\;2,\;\;\;1)\\(\;\;\;0,\;\;\;1,\;\;\;1)
      \\
      (\;\;\;1,-1,-1)\\(\;\;\;1,-1,\;\;\;0)\\(\;\;\;1,\;\;\;0,-1)\\(\;\;\;2,-1,-1)
    \end{matrix}
  \right\} ,
\qquad
v_2\in
  \left\{
    \begin{matrix}
      (\;\;\;1,-1,\;\;\;1)\\(\;\;\;1,-1,\;\;\;2)\\(\;\;\;2,-1,\;\;\;1)\\(\;\;\;1,\;\;\;0,\;\;\;1)
      \\
      (-1,\;\;\;1,-1)\\(-1,\;\;\;1,\;\;\;0)\\(\;\;\;0,\;\;\;1,-1)\\(-1,\;\;\;2,-1)
    \end{matrix}
  \right\} ,
\qquad
v_3\in
  \left\{
    \begin{matrix}
      (\;\;\;1,\;\;\;1,-1)\\(\;\;\;1,\;\;\;2,-1)\\(\;\;\;2,\;\;\;1,-1)\\(\;\;\;1,\;\;\;1,\;\;\;0)
      \\
      (-1,-1,\;\;\;1)\\(-1,\;\;\;0,\;\;\;1)\\(\;\;\;0,-1,\;\;\;1)\\(-1,-1,\;\;\;2)
    \end{matrix}
  \right\} 
.\]
\end{corollary}

\bigskip

So, we now assume that $P$ is boxed with respect to $Q= [0,1]^3$, so that $f_1=x$, $f_2=y$ and $f_3=z$. 
In particular,
\[
v_1=(a_1,\lambda^1_y, \lambda^1_z),
v_2=(\lambda^2_x, a_2,\lambda^2_z),
v_3=(\lambda^3_x, \lambda^3_y, a_3),
\]
where $\lambda^i_*\in \{0,1\}$ 
for all $i\in \{1,2,3\}$, and $a_i\in \Z\setminus \{0,1\}$.

Our final result in this section says that in these conditions each $a_i$ lies within $[-6,7]$.
It relies on Lemmas~\ref{lemma:bound-dim3-coplanar} and~\ref{lemma:bound-dim3}, whose proofs are quite technical and are postponed to Section~\ref{sec:lemmas} in order not to interrupt the flow of reading:

\begin{theorem}
\label{thm:boxed_wrt_unit_cube}
Let $P$ be a lattice $3$-polytope boxed with respect to the unit cube $[0,1]^3$ and of size at least seven. 
Then, with the notations above, $a_i\in \{-6,-5,-4,-3,-2,$ $-1,2,3,4,5,6,7\}$, for all $i$.
\end{theorem}

\begin{proof} 
Let $i \in \{1,2,3\}$. If the edge $s_i$ of $[0,1]^3$ contains some lattice point of $P$, then Lemma~\ref{lemma:firstpoint} implies that $a_i \in \{-1,2\}$. 

Assume this is not the case. That is, the edge $s_i$ of $[0,1]^3$ does not contain lattice points of $P$.
Remember that, since $P$ has size at least seven, then $P \cap Q \cap \Z^3$ consists of at least four lattice points. 
Under these conditions, if $\conv(P \cap Q \cap \Z^3)$ is a facet of $Q$ then Lemma~\ref{lemma:bound-dim3-coplanar} shows that $a_i \in[-6,7]$.
If $P \cap Q \cap \Z^3$ is not contained in a facet of $Q$ then Lemma~\ref{lemma:bound-dim3} shows that $a_i \in[-4,5]$.\qed
\end{proof}

\subsection{Enumeration of boxed $3$-polytopes}
\label{subsec:computer-routines}
We here explain how we combine the results from Sections~\ref{sec:parallelepipedQ} and \ref{subsec:v_i} to computationally enumerate boxed $3$-polytopes of size at least seven. Let $P$ be a boxed $3$-polytope of size at least $7$, so that $P\cap \Z^3=A\cup \{v_1,v_2,v_3\}$ and $A$, of size at least four, is a subset of vertices of a rational parallelepiped $Q$.

\begin{enumerate}

\item \emph{If $Q$ is not the unit cube}, then by Lemma~\ref{lemma:cubes} $Q=Q_0$ and, by Corollary~\ref{coro:vertices_Q_0}, there are at most $8 \times 8 \times 8=512$ possibilities to check for $P$. 
Doing so we find that:

\begin{proposition}
All boxed $3$-polytopes are boxed with respect to the unit cube.
\end{proposition}

\begin{proof}
Checking the $512$ possibilities of Corollary~\ref{coro:vertices_Q_0} we find that there are only five non-isomorphic lattice $3$-polytopes of size seven that are boxed with respect to $Q_0$ and that the five of them are also boxed with respect to the unit cube. 
For size less than seven we use the full list of lattice $3$-polytopes of width larger than one and sizes five or six, contained in~\cite{5points,6points}. Again, it turns out that all the boxed polytopes in those lists are boxed with respect to the unit cube.\qed
\end{proof}

\item \emph{If $Q$ is the unit cube $[0,1]^3$ and $A$ meets every edge of it} then by Corollary~\ref{coro:vertices_Q_0} we know that all $v_i$'s have the $i$-th coordinate in $\{-1,2\}$. We could enumerate all possibilities and check boxedness one by one, but the following lemma allows us to do better:

\begin{lemma}
\label{lemma:boxed-containment}
Let $P\subseteq \R^d$ be boxed with respect to the unit cube $Q=[0,1]^d$ and such that every edge of $Q$ contains at least one lattice point of $P$. Suppose that the size of $P$ is not $2^d+d$ (that is $Q\not\subseteq P$).
Then, for any $u\in \{0,1\}^d\setminus P$, $\conv(P\cup\{u\})$ is boxed with respect to $Q$ and it has size one more than $P$ (that is, $u$ is the only new lattice point).
\end{lemma}

\begin{proof}
By Corollary~\ref{coro:firstpoint}, $P$ is contained in $[-1,2]^d$. 
Let $u\in\{0,1\}^d \setminus P$ and let $P':=\conv(P\cup \{u\}) \subsetneq [-1,2]^d$. Trivially, $P'$ is also boxed with respect to $Q$. It remains to see that $P'\cap \Z^d= P \cap \Z^d\cup \{u\}$. 

Assume the contrary, and let $q \in P' \cap \Z^d \setminus (P\cup \{u\})$. Since $u \in (-1,2)^d$ and $q \in P'\setminus P$, we have that $q\in (-1,2)^d\cap \Z^d=\{0,1\}^d$. 
Now, since $u\not\in P$ but $P$ contains (at least) one point on every edge of $Q$, $P$ contains all the neighbors of $u$ in $Q$. In particular the segment $uq$ intersects $P$, which is a contradiction.\qed
\end{proof}
\smallskip

Thus, in order to enumerate boxed $3$-polytopes of this type we can:
\begin{itemize}
\item Start with the maximal ones, in which $A$ has size eight and we have a priori $8^3=512$ possibilities for $v_1$, $v_2$ and $v_3$ by Corollary~\ref{coro:firstpoint}. Among these possibilities, eliminate redundancies.
\item Remove vertices of $P$ that belong to $Q$ one by one, in all possible manners.
Discard polytopes of width one and eliminate redundancies.
\end{itemize}
This procedure gives us the following numbers of boxed $3$-polytopes:
\medskip

\centerline{
\small
\begin{tabular}{c|ccccccc|c}
{\bf \# vertices}&4 & 5 & 6&7&8&9&10&Total\\
\hline
{\bf size $7$}& 1 &   21 &   28 &    0 &      &     &       & 50 \\
{\bf size $8$}& 2  &  11  &  48  &  30  &   0 &      &       & 91 \\
{\bf size $9$}& 0 & 5  &  24 &  45 &   16  &   0  &   & 90\\
{\bf size $10$}& 1  &   0 &    7  &  21  &  20 &    6 &    0    & 55 \\
{\bf size $11$}&0  &    1  &   0  &   4  &   6  &   4  &   1   & 16  \\
\end{tabular}
}
\medskip

\item \emph{If $Q$ is the unit cube and $A$ does not meet some edge of it}, there are eight possibilities (modulo symmetry) for $A$: one of size six (vertices of a triangular prism), two of size five (vertices of a square pyramid, where the two possibilities are determined by whether the four vertices of the base are in the same facet of $Q$ or not), and five of size four (two coplanarities plus the three types of unimodular tetrahedra in the unit cube).

We then exhaust all the possible coordinates for the vertices $v_i$ which are, according to Theorem~\ref{thm:boxed_wrt_unit_cube}, less than $(2\times2\times 12)^3$ since the $i$-th coordinate of $v_i$ is in $\{-6,\dots,-1,2,\dots,7\}$ and the other two coordinates are in $\{0,1\}$. (This is a huge overcount, since the twelve possibilities have to be considered only when $A$ does not meet the particular edge of $Q$ contained in the same line as $v_i$, which happens quite rarely). This results in the following counts of boxed $3$-polytopes:
\medskip

\centerline{
\small
\begin{tabular}{c|ccccc|c}
{\bf \# vertices}&4 & 5 & 6&7&8 & Total\\
\hline
{\bf size $7$} &4&    51&    47  &   0  &    & 102\\
{\bf size $8$}& 2 &   19 &   72  &  31 &    0 & 124\\
{\bf size $9$}& 0 &    3   & 20  &  35  &   8  & 66 \\ 
\end{tabular}
}
\end{enumerate}

\medskip

Cases (2) and (3) contain some redundancy, since the same configuration can be boxed in more than one way. The following is the irredundant classification of boxed $3$-polytopes by size and number of vertices:
\medskip

\centerline{
\small
\begin{tabular}{c|ccccccc|c}
{\bf \# vertices}&4 & 5 & 6&7&8 &9&10& Total\\
\hline
{\bf size $7$} &4&    51&    49  &   0  &    & &&104\\
{\bf size $8$}& 2 &   19 &   77  &  38 &    0 &&& 136\\
{\bf size $9$}& 0 &    5   & 30  &  56  &   18  & 0&& 279 \\ 
{\bf size $10$}& 1 &    0   & 7  &  21  &   20 & 6  &0& 55 \\ 
{\bf size $11$}& 0 &    1   & 0  &  4  &   6  & 4 & 1 & 16 \\ 
\end{tabular}
}\medskip

Only 32 of these 590 boxed $3$-polytopes are quasi-minimal. These are the numbers of them, in terms of their number of lattice points and vertices:
\medskip

\centerline{
\small
\begin{tabular}{c|ccc|c}
{\bf \# vertices}&4 & 5 & 6 & Total\\
\hline
{\bf size $7$} &4&    15&    4      &23\\
{\bf size $8$}& 2 &   5 &   0     & 7\\
{\bf size $9$}& 0 &    1   &  0   & 1 \\ 
{\bf size $10$}& 1 & 0    & 0   & 1 \\ 
\end{tabular}
}
\medskip

The following matrices, with columns corresponding to vertices, are representatives for them. 


%


\textbf{
\hfil Size 7 \vskip-1cm  \hfil
}
\begin{multicols}{3} 
\footnotesize
\[\left( \begin{array}{cccc}
0 & 1 & 1 & 2\\ 
0 & 1 & 2 & 0\\ 
0 & 2 & 0 & 0\end{array} \right)\]

\[\left( \begin{array}{cccc}
0 & 0 & 1 & 2\\ 
-1& 1 & 1 & 0\\ 
0 & 0 & 3 & 0\end{array} \right)\]

\[\left( \begin{array}{cccc}
0 & 0 & 1 & 2\\ 
-1& 1 & 1 & 0\\ 
0 & 0 & 4 & 0\end{array} \right)\]

\[\left( \begin{array}{cccc}
0 & 0 & 1 & 2\\ 
-1& 1 & 1 & 0\\ 
0 & 0 & 5 & 0\end{array} \right)\]

\[\left( \begin{array}{ccccc}
-1  & 0 & 0 & 1 & 1\\ 
 2 &  0&  0&  -1& 1\\ 
 1 &  0&  1&  1& -1\end{array} \right)\]

\[\left( \begin{array}{ccccc}
 0 & 0 & 1 & 1 & 2\\ 
 1 & 2 & 0 & 0 & 1\\ 
 0 & 1 & 1 & 2 & 0\end{array} \right)\]

\[\left( \begin{array}{ccccc}
 0 & 0 & 1 & 1 & 2\\ 
 1 & 1 & 0 & 2 & 0\\ 
 0 & 2 & 0 & 0 & 0\end{array} \right)\]

\[\left( \begin{array}{ccccc}
 0 & 0 & 0 & 1 & 2\\ 
 0 & 1 & 1 & 2 & 0\\ 
 0 & 1 & 2 & 0 & 0\end{array} \right)\]

\[\left( \begin{array}{ccccc}
 0 & 1 & 1 & 1 & 2\\ 
 1 & 0 & 1 & 2 & 0\\ 
 1 & 0 & 2 & 0 & 0\end{array} \right)\]

\[\left( \begin{array}{ccccc}
 0 & 1 & 1 & 1 & 2\\ 
 1 & 0 & 1 & 2 & 0\\ 
 0 & 0 & 2 & 0 & 0\end{array} \right)\]

\[\left( \begin{array}{ccccc}
 0 & 1 & 1 & 1 & 2\\ 
 0 & 0 & 1 & 2 & 0\\ 
 1 & 0 & 2 & 0 & 0\end{array} \right)\]

\[\left( \begin{array}{ccccc}
 0 & 0 & 1 & 1 & 2\\ 
 0 & 2 & 1 & 1 & 1\\ 
 1 & 1 &  0& 2 & 0\end{array} \right)\]

\[\left( \begin{array}{ccccc}
 0 & 0 & 1 & 1 & 2\\ 
 0 & 2 & 0 & 1 & 0\\ 
 1 & 0 & 0 & 2 & 0\end{array} \right)\]

\[\left( \begin{array}{ccccc}
0  & 0 & 0 & 1 & 2\\ 
 0 & 1 & 1 & 2 & 0\\ 
 1 & 0 & 2 & 0 & 0\end{array} \right)\]

\[\left( \begin{array}{ccccc}
 0 & 1 & 1 & 1 & 2\\ 
 1 & 0 & 1 & 2 & 0\\ 
 0 & 1 & 2 & 0 & 0\end{array} \right)\]

\[\left( \begin{array}{ccccc}
 0 & 0 & 1 & 1 & 2\\ 
 1 & 2 & 0 & 1 & 1\\ 
 0 & 1 & 1 & 2 & 0\end{array} \right)\]

\[\left( \begin{array}{ccccc}
 0 & 1 & 1 & 1 & 2\\ 
 1 & 0 & 1 & 2 & 1\\ 
 1 & 0 & 2 & 0 & 0\end{array} \right)\]

\[\left( \begin{array}{ccccc}
 0 & 1 & 1 & 1 & 2\\ 
 1 & 0 & 1 & 2 & 1\\ 
 0 & 0 & 2 & 0 & 0\end{array} \right)\]

\[\left( \begin{array}{ccccc}
 0 & 1 & 1 & 1 & 2\\ 
 1 & 0 & 1 & 2 & 1\\ 
 0 & 1 & 2 & 1 & 0\end{array} \right)\]

\[\left( \begin{array}{cccccc}
 0 & 0 & 0 & 0 & 1 & 2\\ 
 0 & 1 & 1 & 2 & 0 & 0\\ 
 1 & 0 & 2 & 1 & 0 & 0\end{array} \right)\]

\[\left( \begin{array}{cccccc}
 0 & 1 & 1 & 1 & 1 & 2\\ 
 1 & 0 & 1 & 1 & 2 & 1\\ 
 1 & 1 & 0 & 2 & 0 & 0\end{array} \right)\]

\[\left( \begin{array}{cccccc}
 0 & 1 & 1 & 1 & 1 & 2\\ 
 1 & 0 &  1& 1 & 2 & 1\\ 
 1 & 1 & 0 & 2 & 1 & 0\end{array} \right)\]

\[\left( \begin{array}{cccccc}
 0 & 1 & 1 & 1 & 1 & 2\\ 
 1 & 0 & 1 & 1 & 2 & 1\\ 
 1 & 1 & 0 & 2 & 1 & 1\end{array} \right)\]

\end{multicols}

\textbf{
\hfil Size 8 \vskip-1cm  \hfil
}

\begin{multicols}{4}
\footnotesize

\[\left( \begin{array}{cccc}
0 & 2 & 0 & 1\\ 
0 & 0 & 2 & 0\\ 
0 & 0 &  0& 2\end{array} \right)\]

\[\left( \begin{array}{ccccc}
0  &1  &1  &1  &2 \\ 
 0 & 0 & 1 & 2 & 0\\ 
 0 & 1 & 2 & 0 & 0\end{array} \right)\]

\[\left( \begin{array}{cccc}
0 & 2 & 1 & 1\\ 
0 & 0 & 2 & 0\\ 
0 & 0 & 0 & 2\end{array} \right)\]

\[\left( \begin{array}{ccccc}
 0 & 1 & 1 & 1 & 2\\ 
 1 & 0 & 0 & 2 & 1\\ 
 0 & 0 & 2 & 0 & 0\end{array} \right)\]

\[\left( \begin{array}{ccccc}
0  & 1 & 1 & 1 & 2\\ 
1  & 0 & 0 & 2 & 0\\ 
1  & 0 & 2 & 0 & 0\end{array} \right)\]

\[\left( \begin{array}{ccccc}
 0 & 0 & 0 & 1 & 2\\ 
 0 & 1 & 2 & 1 & 1\\ 
 1 & 0 & 1 & 2 & 0\end{array} \right)\]

\[\left( \begin{array}{ccccc}
 0 & 1 & 1 & 1 & 2\\ 
 1 & 0 & 0 & 2 & 0\\ 
 0 & 0 & 2 &0  &0 \end{array} \right)\]

\end{multicols}

\begin{multicols}{2}

\centerline{\textbf{Size 9}}

\footnotesize
\[\left( \begin{array}{ccccc}
 0 & 2 & 0 & 1 & 1\\ 
 0 & 0 & 2 & 0 & 1\\ 
 0 & 0 & 0 & 2 & 1\end{array} \right)\]

\columnbreak

\normalsize

\centerline{\textbf{Size 10}}

\footnotesize

\[\left( \begin{array}{cccc}
0 & 2 & 0 & 0\\ 
0 & 0 & 2 & 0\\ 
0 & 0 & 0 & 2\end{array} \right)\]
\end{multicols}


\subsection{Technical lemmas for the proof of Theorem~\ref{thm:boxed_wrt_unit_cube}}
\label{sec:lemmas}

We here prove the lemmas that lead to the bound on the $a_i$'s stated in Theorem~\ref{thm:boxed_wrt_unit_cube}. 
For the sake of symmetry, rather than looking at $a_i$ we look at the \emph{distance from $v_i$ to the unit cube}, which we define to be $d_i:=\max\{a_i-1, -a_i\}= |a_i-1/2|-1/2$ (the lattice distance measured with functional $f_i$).
 
Remember that, for each $i\in \{1,2,3\}$, we denote by $r_i$ the unique lattice line that contains $v_i$ and an edge of $Q=[0,1]^3$, and that we will denote said edge as $s_i:=Q\cap r_i$.

\begin{lemma}
\label{lemma:bound-dim3-coplanar}
Let $P$ be a lattice $3$-polytope, boxed with respect to the unit cube $Q=[0,1]^3$, such that $\conv(P \cap Q \cap \Z^3)$ is a facet of $Q$. Then $d_i \le 6$ for all $i$.
\end{lemma}

\begin{proof}
As usual, let $(P\setminus [0,1]^3 ) \cap \Z^3=\{v_1,v_2,v_3\}$.
Since $\conv(P\cap Q\cap\Z^3)$ is a facet of $Q$, $P\cap Q \cap \Z^3$ has four lattice points and $P$ has size seven. Without loss of generality 
\[
A_0:=P \cap \{0,1\}^3 =\{(0,0,0), (1,0,0), (0,1,0), (1,1,0)\},
\]
By symmetry of the assumptions made so far with respect to the planes $\{x=1/2\}$ and $\{y=1/2\}$, we can take without loss of generality $v_3=(0,0,a_3)$. By Lemma~\ref{lemma:firstpoint}, $a_3$ must be either $-1$ or $2$, but if $v_3$ were $(0,0,2)$ then the point $(0,0,1)$ would be in $P$, which contradicts the assumptions. Thus, we assume that $v_3=(0,0,-1)$ for the rest of the proof. 

So far we have that $d_3=1$. We need to prove that the value for $d_1$ and $d_2$ is bounded by $6$. 

Vertices $v_1$ and $v_2$ have the third coordinate in $\{0,1\}$. In order for $P$ not to have width one with respect to the functional $z$, at least one of $v_1$ and $v_2$ must lie in the plane $\{z=1\}$. Two things can happen: either the two vertices are in the plane $\{z=1\}$, or there is one in $\{z=1\}$ and another in $\{z=0\}$. 

Let us see, in both cases, what conditions on $v_1$ and $v_2$ are necessary for no extra lattice points to arise in the plane $\{z=0\}$. This technique is very close to what we called the \emph{parallel planes method} in~\cite{6points}. The main idea is that if we have a lattice $3$-polytope $P$ contained in $\R^2\times [-1,1]$ and the intersection of $P$ with the planes $\R^2\times \{-1\}$ and $\R^2\times \{1\}$ is easy to understand (in our case it will be a point or a segment) then the only additional lattice points that can arise will be in the intersection of $P$ with $\R^2\times \{0\}$, and this intersection equals the convex hull of $P \cap (\Z^2\times \{0\})$ together with the mid-points of edges joining the vertices of $P$ in the other two planes.
\smallskip

\begin{enumerate}

\item[(a)] One point in each plane. Without loss of generality, since the conditions on $v_1$ and $v_2$ are symmetric under the exchange of $x$ and $y$:
\[v_1=(a_1,\lambda^1_y,1), \qquad v_2=(\lambda^2_x,a_2,0),\]
with $a_1 \in \Z$, $a_2\in \{-1,2\}$ (by Lemma~\ref{lemma:firstpoint}, since $s_2 \subseteq \{z=0\}$ contains lattice points of $P$) and $\lambda^i_* \in \{0,1\}$. In particular $d_2=1$. Since the conditions so far are symmetric under $(x,y,z)\mapsto(1-x+z,y,z)$ and this symmetry exchanges the two possible values of $\lambda^2_x$, we can further assume that $\lambda^2_x=0$ and hence
\[
v_2\in \{(0,-1,0), (0,2,0)\}.
\]

Let us see which values are allowed for the coordinates of $v_1$ so that $P$ has no extra lattice point. Observe that since $P$ is contained in the region $z\in[-1,1]$ and has only one vertex at each $\{z=\pm 1\}$, extra lattice points can only arise in the plane $\{z=0\}$. The intersection of $P$ with the plane $\{z=0\}$ equals the convex hull of $A_0 \cup \{v_2, v_1'\}$, where $v'_1$ is the intersection point of the edge $v_1v_3$ with that plane. This intersection point is 
\[
v_1'=\left(\frac{a_1}{2},\frac{\lambda^1_y}{2},0\right) \in \frac{1}{2}\Z  \times \left\{0,\frac{1}{2}\right\} \times \{0\}. 
\]

\begin{figure}[h]
\centerline{
\includegraphics[scale=0.7]{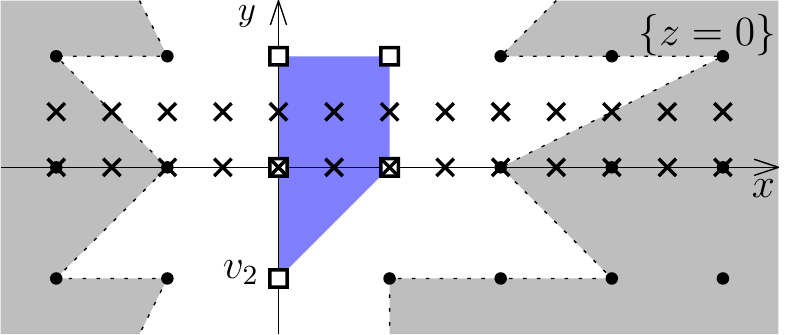} \qquad
\includegraphics[scale=0.7]{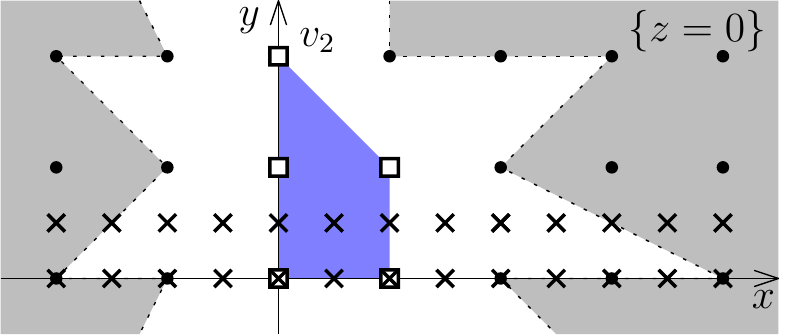} 
}
\centerline{
Case $v_2=(0,-1,0)$ \qquad \qquad \qquad \qquad  \qquad \qquad Case $v_2=(0,2,0)$
}
\caption{The possible positions for $v_1'$ (hence for $v_1$) in case (a) in the proof of Lemma~\ref{lemma:bound-dim3-coplanar}. Black dots represent lattice points. White squares represent lattice points of $P \cap \{z=0\}$. Crosses mark the positions for $v_1'$ corresponding to $v_1$ lying in its chimney. The white (open) region are the positions where a point $v'_1$ can be placed with the property that $\conv(A_0\cup \{v_2,v'_1\})$ does not have extra lattice points. The intersection of both gives the valid positions for $v'_1$. The blue area is $\conv(A_0 \cup \{v_2\})$.} 
\label{fig:3dim_coplanar_1}
\end{figure}

Figure~\ref{fig:3dim_coplanar_1} shows that in order for no extra lattice points to arise we must have $a_1/2 \in [-1,5/2]$ so that $a_1\in [-2,5]$. That is, $d_1\le 4$.

\medskip

\item[(b)] Both points in the plane $\{z=1\}$:
\[v_1=(a_1,\lambda^1_y,1), \qquad v_2=(\lambda^2_x,a_2,1)\]
with $a_i \in \Z$ and $\lambda^i_* \in \{0,1\}$.

Let us see which values of $a_i$ and $\lambda^i_*$ are allowed so that no extra lattice point is added when considering the whole polytope. $P$ is contained in the region $\{z\in[-1,1]\}$ and has a point $v_3$ in $\{z=-1\}$ and two points ($v_1$ and $v_2$) in $\{z=1\}$. Thus, the intersection of $P$ with $\{z=0\}$ equals $\conv(A_0 \cup\{v'_1,v'_2\})$ where $v'_1$ and $v'_2$ are the intersection points of the edges $v_1v_3$ and $v_2v_3$ with that plane. Namely:
\[
v_1'=\left(\frac{a_1}{2},\frac{\lambda^1_y}{2},0\right) \in \frac{1}{2}\Z \times \left\{0,\frac{1}{2}\right\} \times \{0\} 
\]
and
\[
v_2'=\left(\frac{\lambda^2_x}{2},\frac{a_2}{2},0\right) \in  \left\{0,\frac{1}{2}\right\} \times\frac{1}{2}\Z \times \{0\}. 
\]

We are going to do a case-study based on the four possibilities for $\lambda^1_y$ and $\lambda^2_x$. 
Let us first see that, if $\lambda^i_*=0$, then $a_i\in\{-1,2,3\}$, for $i=1,2$. Suppose $\lambda^1_y=0$. Then the intersection of $P$ with the plane $\{y=0\}$ is the convex hull of the points $(0,0,0)$, $(1,0,0)$, $v_3=(0,0,-1)$ and $v_1=(a_1,0,1)$. In order for $(2,0,0)$ and $(-1,0,0)$ not to lie in $P \cap \{y=0\}$, the value of $a_1$ is restricted to $(-2,4)$. 
See the following figure, where the black squares are the possibilities for $v_1$:
\medskip 

\centerline{\includegraphics[scale=.9]{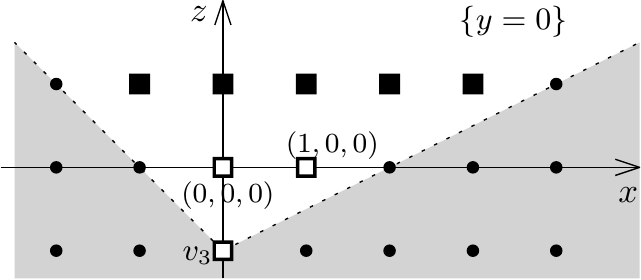}}
\medskip

Since $a_1\in\Z\setminus\{0,1\}$, then $a_1\in\{-1,2,3\}$. By symmetry under the exchange of $x$ and $y$, the same happens for $a_2$ if $\lambda^2_x=0$. Then:
\smallskip 

\begin{itemize}
\item If $\lambda^1_y=\lambda^2_x=0$, then $a_1,a_2\in \{-1,2,3\}$, so $d_1,d_2\le 2$.
\smallskip

\item If $\lambda^1_y=1$ and $\lambda^2_x=0$, then $a_2\in \{-1,2,3\}$ and we need to look at possible values of $a_1$. We can discard the case $a_2=2$, because then $P$ has width one with respect to the functional $y-z$. Since the conditions so far are symmetric under $(x,y,z)\mapsto(x,1-y+z,z)$ and this symmetry exchanges $(0,-1,1)$ and $(0,3,1)$, we can assume that $v_2=(0,3,1)$ (hence $v'_2=(0,\frac32,0)$).

\begin{figure}[ht]
\centerline{\includegraphics[scale=0.8]{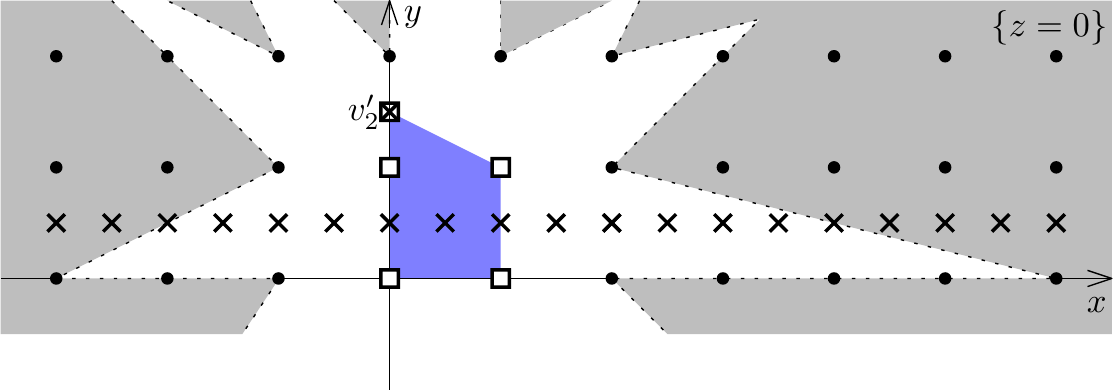} }
\caption{The possible positions for $v_1'$ (hence for $v_1$) in case (b), for $\lambda^1_y=1$ and $\lambda^2_x=0$, in the proof of Lemma~\ref{lemma:bound-dim3-coplanar}. Black dots represent lattice points. White squares represent lattice points of $P \cap \{z=0\}$. The crossed square is the rational point $v'_2$. Crosses mark the positions for $v_1'\in\R^2\times\{0\}$ corresponding to $v_1$ lying in its chimney. The white (open) region are the positions where a point $v'$ has the property that $\conv(A_0\cup \{v_2',v'\})$ does not have extra lattice points. The intersection of both gives the valid positions for $v'_1$. The blue area is $\conv(A_0 \cup \{v'_2\})$.} 
\label{fig:3dim_coplanar_2}
\end{figure}

The admissible positions of $v_1$ (or, rather, of $v'_1$) are drawn in Figure~\ref{fig:3dim_coplanar_2}. As seen in the figure, the valid positions of $v'_1$ have first coordinate $a_1/2 \in [-3/2,7/2]$ so that $a_1\in [-3,7]$ (notice that the symmetry $(x,y,z)\mapsto(x,1-y+z,z)$ fixes $v_1=(a_1,1,1)$). That is, $d_1\le 6$ and $d_2\le 2$.
\smallskip

\item The case $\lambda^1_y=0$ and $\lambda^2_x=1$ is symmetric to the previous one, so it leads to $a_1\in \{-1,2,3\}$ and $a_2\in [-3,7]$. That is, $d_1\le 2$ and $d_2\le 6$.

\smallskip

\item If $\lambda^1_y=\lambda^2_x=1$ we have $v_1=(a_1,1,1)$ and $v_2=(1,a_2,1)$ with $a_i\in \Z\setminus \{0,1\}$. Notice that we can also assume that $a_i \neq 2$ (if $a_1=2$ $P$ has width one with respect to $x-z$ and if $a_2=2$ it has width one with respect to $y-z$).
In this case, the conditions so far on the configuration are symmetric under both $(x,y,z)\mapsto(1-x+z,y,z)$ and $(x,y,z)\mapsto(x,1-y+z,z)$, which reflect $v_1$ and $v_2$ within their respective chimneys. Hence we can assume that both $v_1$ and $v_2$ lie in their respective positive half-chimneys, that is, $a_1,a_2>2$.

In the plane $\{z=0\}$, we have now that 
\[
v_1'=\left(a'_1,\frac{1}{2},0\right),
\quad
v_2'=\left(\frac{1}{2},a'_2,0\right),
\]
where $a'_i=a_i/2>1$.

The crosses in Figure~\ref{fig:3dim_coplanar_3} show the possible positions for $v'_1$ and $v'_2$.

\begin{figure}[ht]
\centerline{\includegraphics{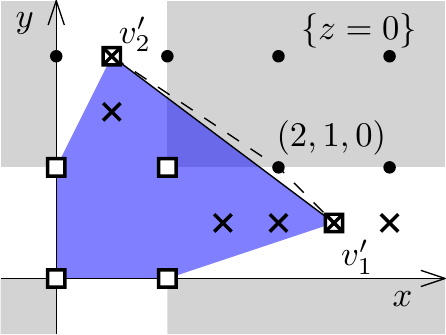}}
\caption{The possible positions for $v_i'$ (hence for $v_i$) in case (b), for $\lambda^1_y=1=\lambda^2_x$, in the proof of Lemma~\ref{lemma:bound-dim3-coplanar}. Black dots represent lattice points. White squares represent lattice points of $P \cap \{z=0\}$. Crosses mark the possible positions for $v_i'$ corresponding to $v_i$ lying in their chimneys (the white regions). The crossed squares represent specific choices for the rational points $v'_i$. The segment $v'_1v'_2$ must separate $(2,1,0)$ from $Q$. The blue area is $P \cap \{z=0\}$.} 
\label{fig:3dim_coplanar_3}
\end{figure}

Then, in order for the point $(2,1,0)$ not to be in $P$ we need the triangle $\conv\{(2,1),(a'_1,1/2),(1/2,a'_2)\}$ to be negatively oriented, which amounts to:
\[
\left \vert
\begin{matrix}
1&2&1\\
1&a'_1&1/2\\
1&1/2&a'_2\\
\end{matrix}
\right \vert
=a'_1a'_2+\frac54 - a'_1-2a'_2 <0.
\]
Equivalently,
\[
a'_1(a'_2-1) < 2a'_2 -\frac54.
\]
Since $a'_2\ge 3/2$, this is the same as
\[
a'_1 <\frac{2a'_2-\frac54}{a'_2-1} = 2 +\frac{\frac34}{a'_2-1} \le \frac72.
\]
The same arguments using the point $(1,2,0)$ and that $a'_1\ge 3/2$ (or simply the symmetry $x\leftrightarrow y$) give $a'_2<\frac72$. Hence $a_1, a_2 \le 6$. That is, $d_1,d_2\le 5$.

\end{itemize}
\end{enumerate}

Summing up, an upper bound for the distance is $d_i \le 6$ for all $i$.\qed
\end{proof}

For the remaining case we first study a similar question for boxed $2$-polytopes.

\begin{lemma}
\label{lemma:bound-dim2}
Let $P$ be a lattice polygon, boxed with respect to the unit square $Q=[0,1]^2$. Suppose that $P$ intersects the edge of $Q$ contained in $\{x=0\}$ and does not contain the vertices $(1,0)$ and $(1,1)$ of $Q$. Assume further that $v_2=(1,a_2)$. Then, $a_2\in \{-2,-1,2,3\}$.
 \end{lemma}

\begin{proof}
We assume without loss of generality that $a_2 >1$ and want to show that $a_2<4$ (the case $a_2<0$ is symmetric with respect to the line $\{y=1/2\}$).
Let $q_0=(0,y_0)\in P$, with $y_0 \in [0,1]$, be the point guaranteed by the hypotheses. We distinguish according to the possible positions of $v_1=(a_1, \lambda^1_y)$. Remember that $\lambda^1_y\in\{0,1\}$ and $a_1\in \Z\setminus \{0,1\}$:
\begin{itemize}
\item If $a_1>1$, then $P$ contains a point $q_1=(1,y_1) \in \conv\{v_1,q_0\}$ with $y_1\in[0,1]$. But then, the point $(1,1)$ must be in $P$ (a contradiction) since it lies in the segment $q_1 v_2$ (see Figure~\ref{fig:2dim}).

\begin{figure}[ht]
\centerline{
\includegraphics[scale=0.8]{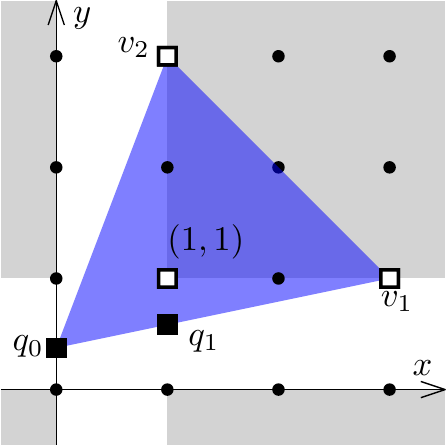} \qquad
\includegraphics[scale=0.8]{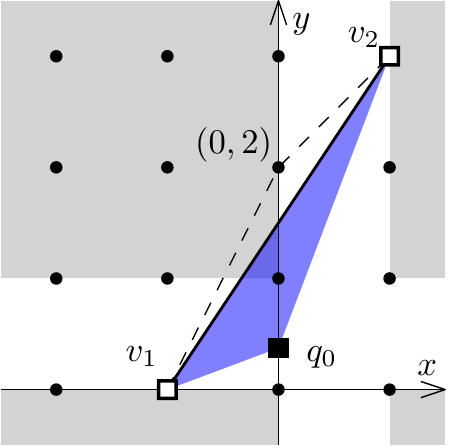}
}
\centerline{
Case $a_1>1$ \qquad  \qquad \qquad \qquad Case $a_1 <0$
}
\caption{The analysis of the cases where $v_1$ lies in $x>1$ or $x<0$ in the proof of Lemma~\ref{lemma:bound-dim2}. Black dots represent lattice points. Black squares represent the (possibly non-integer) points $q_0$ and $q_1$ that lie in $P$. White squares represent the vertices $v_1$ and $v_2$ of $P$ and the lattice point $(1,1)$ when it lies in $P$. Each of the $v_i$ lie in their corresponding chimney (the white regions). The blue area is a (maybe rational) subpolytope of $P$.
}
\label{fig:2dim}
\end{figure}

\item If $a_1 < 0$, consider the segment $v_1v_2$. This segment intersects the line $\{x=0\}$ at height smaller than $2$, or otherwise the point $(0,2)$ is in $P$ (see Figure~\ref{fig:2dim}).

That is, we want the following determinant to be positive:
\[
\begin{array}{|ccccc|}
1&&1&&1\\
a_1&&1&&0\\
\lambda^1_y && a_2 && 2 \\
\end{array}
>0.
\]
Equivalently,
\[
a_1a_2+2-\lambda^1_y-2a_1 >0,
\]
or (since $a_1\le -1$ and $\lambda^1_y\in \{0,1\}$)
\[
a_2 < 2 + \frac{2-\lambda^1_y}{|a_1|} \le 4 - \lambda^1_y \le 4.\qed
\]
\end{itemize}
\end{proof}

\begin{lemma}
\label{lemma:bound-dim3}
Let $P$ be a lattice $3$-polytope, boxed with respect to the unit cube $Q=[0,1]^3$ and of size at least seven. Suppose that:
\begin{itemize}
\item The edge $s_i$ does not contain lattice points of $P$ for some $i$.
\item $P \cap Q \cap \Z^3$ is not contained in a facet of $Q$.
\end{itemize}
Then, $d_i \le 4$.
\end{lemma}

\begin{proof}
Let $i\in\{1,2,3\}$ be such that $s_i$ does not contain lattice points of $P$ and let us prove that $d_i \le 4$. We will prove the case for $i=3$ (the other cases are analogous). Without loss of generality, $v_3=(1,1,a_3)$ and $P$ does not contain the points $(1,1,0),(1,1,1)$.

Since $P \cap Q \cap \Z^3$ has at least four points but is not contained in a facet of $Q$, and since neither $(1,1,0)$ nor $(1,1,1)$ are in $P$, $P \cap Q \cap \Z^3$ contains at least one point from each of $\{(1,0,0),(1,0,1)\}$ and $\{(0,1,0), (0,1,1)\}$. That is, let $q_{10}=(1,0, z_{10})$, $q_{01}=(0,1, z_{01}) \in P$ for some $z_{10}, z_{01} \in \{0,1\}$. We can also assume, without loss of generality, that $a_3>1$ (the case $a_3<0$ is symmetric with respect to the plane $\{y=1/2\}$).\medskip

We distinguish cases according to the possible positions of 
\[
v_1=(a_1, \lambda^1_y,\lambda^1_z), \; v_2=(\lambda^2_x,a_2,\lambda^2_z).
\]
Remember that $\lambda^i_*\in\{0,1\}$ and $a_i\in \Z\setminus \{0,1\}$, for $i=1,2$.\medskip

First, suppose $\lambda^2_x=1$. Then $v_2,v_3,q_{10}\in P$ are in the plane $\{x=1\}$. Then $P \cap \{x=1\}$ is as in the hypothesis of Lemma~\ref{lemma:bound-dim2}, hence $a_3 \in \{2,3\}$. By symmetry of coordinates $x$ and $y$, the same happens if $\lambda^1_y=1$, using the point $q_{01}$. That is, in this case $d_3\le 2$.\medskip

So now we have the case where $\lambda^1_y=0=\lambda^2_x$:
\[
v_1=(a_1, 0,\lambda^1_z), \; v_2=(0,a_2,\lambda^2_z)
\]

We will now separate the cases where $a_1,a_2$ are positive or negative (see Figure~\ref{fig:3dim_123}):

\begin{itemize}

\item If $a_1,a_2>1$, then $P$ contains a point $q_{11}=(1,1,z_{11}) \in \conv\{v_1,v_2,q_{10}\}$ with $z_{11}\in[0,1]$. But then, the point $(1,1,1)$ must be in $P$ (a contradiction) since it lies in the segment $q_{11} v_3$.

\item If $a_1,a_2 < 0$, then $P$ contains a point $q_{00}=(0,0,z_{00}) \in \conv\{v_1,v_2,q_{10}\}$, with $z_{00} \in[0,1]$. Consider now the triangle $v_1v_2v_3$. This triangle intersects the line $\{x=0=y\}$ at a height smaller than $2$, or otherwise, since $q_{00}\in P$, the point $(0,0,2)$ must be in $P$ (a contradiction). That is, we want the following determinant to be positive:
\[
\begin{array}{|cccc|}
1&1&1&1\\
a_1&0&1&0\\
0 &a_2&1&0\\
\lambda^1_z &\lambda^2_z& a_3 & 2 \\
\end{array}
>0.
\]
That is, $-a_1a_2a_3+\lambda^1_za_2+\lambda^2_za_1+2(a_1a_2-a_1-a_2) >0$ or, (since $a_1,a_2\le -1$ and $\lambda^1_z,\lambda^2_z\in \{0,1\}$)

\begin{eqnarray*}
a_3 &<& 2\left(\frac{(-a_1)(-a_2)+(-a_1)(1-\lambda^2_z/2)+(-a_2)(1-\lambda^1_z/2)}{(-a_1)(-a_2)}\right) =\\
&=& 2\left(1+\frac{1-\lambda^2_z/2}{-a_2}+\frac{1-\lambda^1_z/2}{-a_1}\right) \le  2\left(1+\frac{1}{-a_2}+\frac{1}{-a_1}\right) \le 6\\
\end{eqnarray*}
That is, $d_3\le 4$.

\item One is positive and one negative: suppose $a_1>1$ and $a_2 <0$. In this case, the triangle with vertices $v_1,v_2,v_3 \in P$ must intersect the line $\{x=1,y=0\}$ at height smaller than $2$, otherwise, since $q_{10} \in P$, the point $(1,0,2)$ is in $P$. This is equivalent to the following determinant being positive:

\[
\begin{array}{|cccc|}
1&1&1&1\\
0&a_1&1&1\\
a_2&0 &1&0\\
\lambda^2_z &\lambda^1_z& a_3 & 2 \\
\end{array}
>0.
\]
That is, $2a_2 +a_1a_2a_3-2a_1a_2-a_2a_3+2a_1+\lambda^2_z-a_1\lambda^2_z-\lambda^1_z>0$ 
or (since $a_1\ge2$, $a_2\le -1$ and $\lambda^1_z,\lambda^2_z\in \{0,1\}$)

\begin{eqnarray*}
a_3 &<& 2\left(\frac{(-a_2)(a_1-1)+a_1  -\lambda^1_z/2-\lambda^2_z(a_1-1)/2}{(-a_2)(a_1-1)}\right) \le\\
&\le& 2\left(\frac{(-a_2)(a_1-1)+a_1}{(-a_2)(a_1-1)}\right) = \\
&=&2\left(1+\frac{1}{-a_2}+\frac{1}{(-a_2)(a_1-1)}\right) \le 6\\
\end{eqnarray*}

That is, $d_3\le 4$. By the symmetry $x\leftrightarrow y$, the same occurs if $a_1<0$ and $a_2>1$.\qed
\end{itemize}

\begin{figure}[ht]
\centerline{
\includegraphics[scale=0.72]{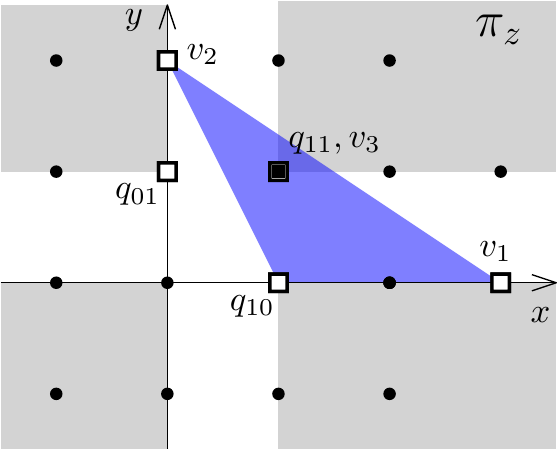} \quad
\includegraphics[scale=0.72]{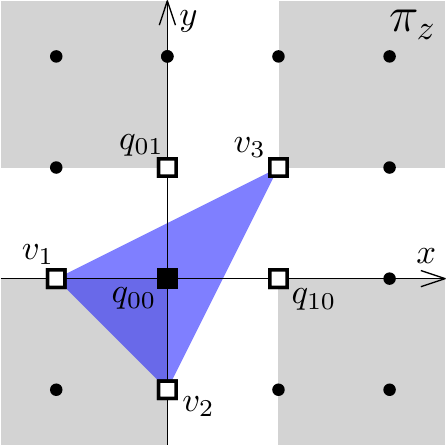} \quad
\includegraphics[scale=0.72]{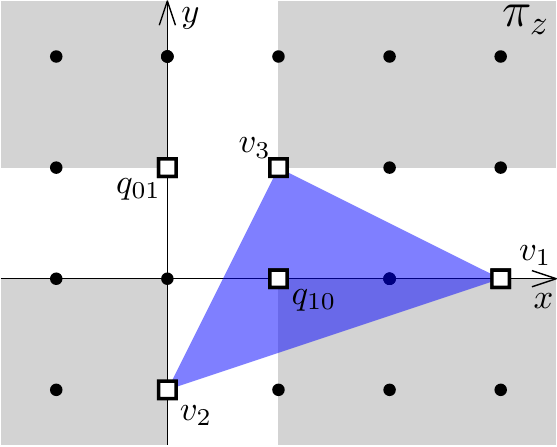}
}
\centerline{
\qquad Case $a_1,a_2>1$ \qquad \qquad\qquad Case $a_1,a_2<0$  \qquad\qquad\quad Case $a_1>1$, $a_2<0$
}
\caption{The analysis of the cases where $v_i$ lies in $C_i^+$ or in $C_i^-$, for $i=1,2$, in the proof of Lemma~\ref{lemma:bound-dim3}. The figures represent the projection in the direction of the third coordinate. Black dots represent lattice points. Black squares represent the (possibly non-integer) points $q_{ij}$ that lie in $P$. White squares represent the vertices $v_i$ and lattice points $q_{ij}$ of $P$. Each of the vertices $v_i$ lies in its corresponding chimney (the white regions). The blue area is a (maybe rational) subpolytope of $P$.
}
\label{fig:3dim_123}
\end{figure}
\end{proof}

\medskip


\section{Results of the enumeration}
\label{sec:result-tables}

The results of Sections~\ref{sec:spiked} and \ref{sec:boxed} allow us to completely enumerate quasi-minimal $3$-polytopes. The counts of them are given in Table~\ref{table-minimal}.
Boxed ones are finitely many and of size at most ten. They are enumerated by computer as explained in Section~\ref{subsec:computer-routines}. 
For spiked ones, the number was computed in Section~\ref{sec:spiked} (see Remark~\ref{remark:spiked}). These two counts are shown in the left and center parts of the table, and the right part contains the union of the two sets. For sizes $5$ and $6$ a polytope can be boxed and spiked at the same time (see Remarks~\ref{remark:spiked} and~\ref{rem:not_boxed}), so we do not give the separate numbers. In fact, the numbers of quasi-minimal $3$-polytopes of these sizes were not computed with the methods of this paper, but directly extracted from the classifications in~\cite{5points,6points}.
 
\vspace{-2ex}


\begin{table}[h]
\footnotesize
\begin{tabular}{cccc}
& boxed & spiked & all
\\
\begin{tabular}{c}
{\bf \# vertices}\\
\hline
{\bf size $5$}\\
{\bf size $6$}\\
{\bf size $7$}\\
{\bf size $8$}\\
{\bf size $9$}\\
{\bf size $10$}\\
{\bf size $11$}\\
{\bf size $>11$}\\
\hline
 \text{\bf Total}
\end{tabular}

&

\begin{tabular}{ccc|c}
\hline
4 & 5 & 6& \text{total}\\
\hline
&&&\\
&&&\\
      4&    15&     4  &23\\
      2&     5&      0&7\\
      0&     1&       0& 1\\ 
    1&   0&         0&1\\
    0&     0&         0&0\\
    0&     0&         0&0\\
\hline
 7 
 & 21 
 & 4 
 & 32
\end{tabular}

&

\begin{tabular}{cc|c}
\hline
 4 & 5 & \text{total}\\
\hline
&&\\
&&\\
      21&    6&       27\\
      22&    13&      35\\
      26&     17&       43\\
    24&     21&         45\\
    24&     25&         49\\
$\infty$ &  $\infty$ & $\infty$ \\
\hline
$\infty$ &  $\infty$ & $\infty$ 
\end{tabular}

&

\begin{tabular}{ccc|c}
 \hline
 4 & 5 & 6& \text{total}\\
\hline
9&  0  &    &9\\
22&13&  0  &35\\
      25&    21&     4  &50\\
      24&     18&     0 &42\\
      26&     18&      0 & 44\\
    25&     21&        0 &46\\
    24&     25&        0 &49\\
$\infty$ &  $\infty$ & $0$ & $\infty$\\
\hline
$\infty$ &  $\infty$ & $4$ & $\infty$
 \end{tabular}
 
 \end{tabular}
\medskip
\caption{Quasi-minimal $3$-polytopes, classified according to their numbers of lattice points (row) and vertices (column). 
For size $n>11$ there are $4n-19$ spiked $3$-polytopes with $5$ vertices and $24$ (or $26$ if $n\equiv0\pmod3$) with $4$ vertices.}
\label{table-minimal}
\end{table}

\vspace{-1ex}

Once we have quasi-minimal $3$-polytopes completely classified, we can run the enumeration algorithm described in the introduction taking as input the list of lattice $3$-polytopes of size six and width larger than one, contained in~\cite{6points}. In the following sections we show the results of this enumeration, that we carried out up to size $11$. The complete lists of these polytopes can be found at \url{http://personales.unican.es/santosf/3polytopes/}.

\subsection{Classification by number of vertices and/or interior points}
 
The summary of our enumeration of lattice $3$-polytopes is given in Table~\ref{table:main}. 
Observe that the zeros in the diagonal ``size=vertices'' follow from Howe's Theorem~\cite{Scarf}: if all lattice points of a lattice $3$-polytope $P$ are vertices then $P$ has width one. We also show the approximate computation times.
%
 \begin{table}[h] \small
\begin{tabular}{c|ccccccc|c|c}
{\bf \# vertices}& 4 & 5 & 6&7& 8&9&10&\text{total} & \hbox{time}\\
\hline
{\bf size $5$} & $9$  & $0$  &       &       &    &     &     &$9$&  from \cite{5points}\\
{\bf size $6$} & $36$  &$40$   & $0$  &       &    &     &     &$76$ &  from \cite{6points}\\
{\bf size $7$} & $103$  & $296$  & $97$  & $0$  &    &     &     &$496$& $14$ mins.\\
{\bf size $8$} & $193$  & $1195$  &  $1140$ & $147$  &$0$&     &     & $2675$   &$70$ mins.\\
{\bf size $9$} & $282$  & $2853$  & $5920$  & $2491$  & $152$ &$0$&     &$11698$& $7$ hours\\
{\bf size $10$}& $478$  & $5985$  & $18505$  & $16384$  & $3575$& $108$&$0$ &$45035$& $48$ hours\\
{\bf size $11$}& $619$  & $11432$  &  $48103$ &  $64256$ & $28570$& $3425$& $59$&$156464$ &$20$ days\\
\end{tabular}
\medskip
\caption{Lattice $3$-polytopes of width larger than one and size $\le 11$, classified according to their size and number of vertices. The computations were made in MATLAB R2014b, on a 2.60 GHz CPU desktop.}
\label{table:main}
\end{table}

\vspace{-5ex}

\begin{remark}
\rm
The total number of lattice $3$-polytopes of width larger than one seems experimentally to grow more slowly than a single exponential. But the only theoretical upper bound that we can derive from the merging algorithm is doubly exponential, which follows from the following recurrence: let $S(n)$ be the number of lattice $3$-polytopes of width larger than one and size $n$, then
\[
S(n+1) \le 24 {n \choose 4} \binom{S(n)+1}2 + 4n+11.
\]
In this formula, the term $4n+11$ is an upper bound for the number of quasi-minimal $3$-polytopes of size $n+1\ge11$; $\binom{S(n)+1}2$ is the number of pairs of polytopes to be tested for merging and $24{n \choose 4}$ is a (crude) upper bound for the number of possible mergings: we choose an (ordered) affine basis in the first polytope and each merging is represented by an affine map sending it to an ordered basis in the second).
\end{remark}

Table~\ref{table:interior} shows a finer classification, in which the number of interior lattice points is also considered. 
Polytopes with a single interior lattice point are of special importance in algebraic geometry, for their connections with toric varieties. They are called \emph{canonical}. If, moreover, all lattice points except for the interior one are vertices then they are called \emph{terminal}. The numbers of canonical and terminal polytopes for each size can be extracted from Table~\ref{table:interior} and are shown in Table~\ref{table:canonical}. Their full classification was previously done by Kasprzyk~\cite{Kasprzyk}, and our results agree with it.

\begin{table}[h]
\begin{tabular}{c|ccccccc}
{\bf Size}& $5$ & $6$& $7$& $8$&$9$ &$10$&$11$\\
\hline
\textbf{ Canonical} & $8$  & $49$  & $218$  & $723$  & $1990$ &$4587$&$9376$\\
\textbf{Terminal} & $8$  & $38$  & $95$  & $144$  & $151$ &$107$&$59$ \\
\end{tabular}
\medskip
\caption{Canonical and terminal $3$-polytopes of size $\le 11$. }
\label{table:canonical}
\end{table}

\subsection{Classification by width}
In Table~\ref{table:width} we present the classification of lattice $3$-polytopes by size and width (Remember that those of width one are infinitely many for each size). There is a remarkable gap between lattice $3$-polytopes of width three, that exist already with six lattice points, and of width four, which need ten lattice points at the least.
It is also worth noting that in every size the maximum width is achieved (perhaps not uniquely) by some clean tetrahedron. Remember that a lattice polytope is \emph{clean} if all its boundary lattice points are vertices.
\vspace{-2ex}


\begin{table}[htb]
\begin{tabular}{c|cccccccc}
{\bf Size}& $4$ & $5$ & $6$&$7$& $8$&$9$&$10$&$11$\\
\hline
{\bf width $2$} & $0$  & $9$  & $74$& $477$&$2524$&$10862$&$40885$&$137803$\\
{\bf width $3$} & 0  &    $0$    & $2$ & $19$  &$151$ &$836$&$4148$&$18635$\\
{\bf width $4$} & 0  &  0      & $0$   & $0$        &  $0$       &$0$&$2$&$26$\\
{\bf width $5$} & 0  &  0      &  0  &  0       & 0        &0&$0$&$0$\\
\end{tabular}
\medskip
\caption{Lattice $3$-polytopes of size $\le 11$, classified by width.}
\label{table:width}
\end{table}

\vspace{-7ex}

\subsection{Volumes of lattice $3$-polytopes}
\label{subsec:results-volumes}

Figure~\ref{fig:volumes} shows the normalized volumes that arise among lattice $3$-polytopes of width larger than one and sizes $5$ to $11$.

\begin{figure}[h]
\centerline{\color{black}
\includegraphics[scale=.8]{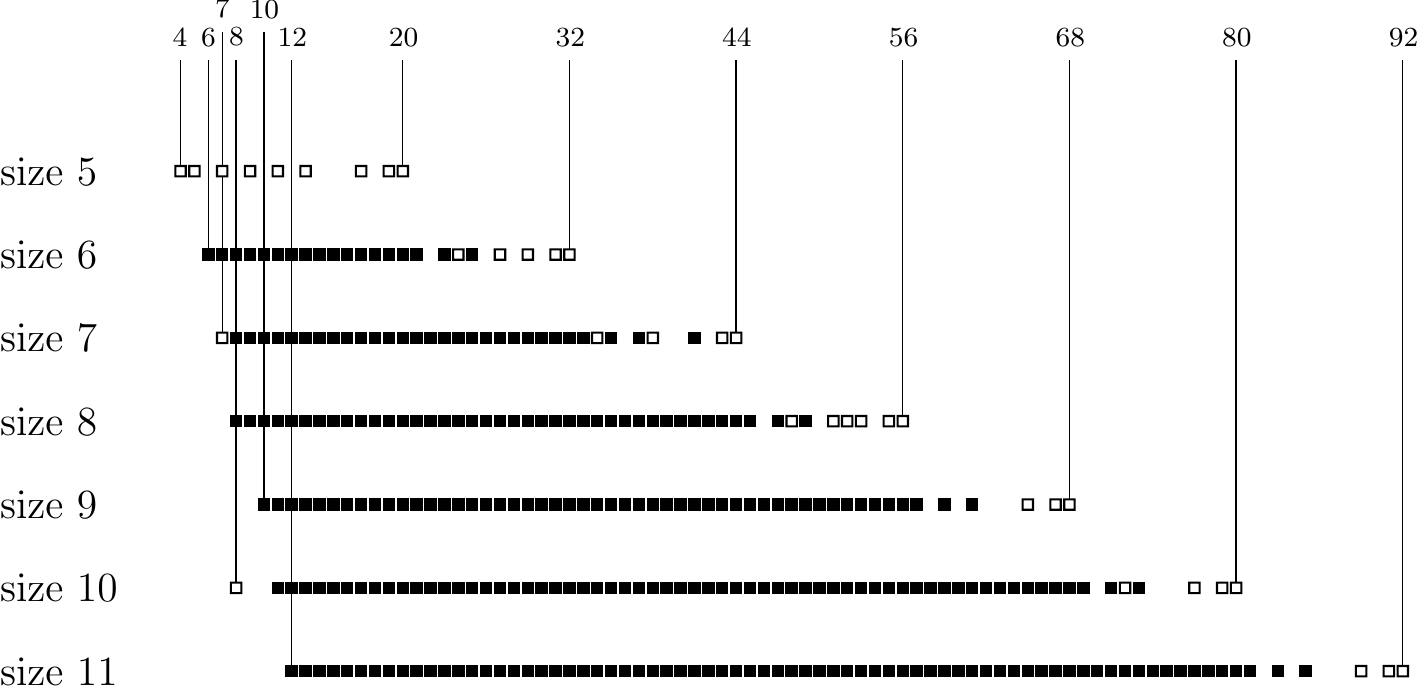}}
\caption{A square in the row $n$ and column $v$ means that there is some lattice $3$-polytope of width $>1$ and size $n$ of (normalized) volume $v$. A white square means that there is exactly one such polytope. The values for $v$ displayed in the top row are the minimum and maximum values achieved for each size.}
\label{fig:volumes}
\end{figure}


\begin{table}[h] \small
\begin{tabular}{cccc}
\begin{tabular}{c}
\\
\hline
{\bf \# vertices} \\
\hline
\textbf{ $0$ int. pts.} \\
\textbf{ $1$ int. pts.} \\
\textbf{ $2$ int. pts.} \\
\textbf{ $3$ int. pts.} \\
\hline
\textbf{ total}           \\
\end{tabular}
&

\begin{tabular}{c|c}
\multicolumn{2}{c}{\textbf{Size $5$}}\\
\hline
{\bf 4} & \textbf{total} \\
\hline
$1$    & $1$   \\
$8$    & $8$   \\
        &        \\
        &        \\
\hline
$9$    & $9$   \\
\end{tabular}
&

\begin{tabular}{cc|c}
\multicolumn{3}{c}{\textbf{Size $6$}}\\
\hline
 {\bf  4} & {\bf 5} &  \textbf{total} \\
\hline
 $2$  & $2$  & $4$     \\
 $11$  & $38$  & $49$     \\
 $23$  &       & $23$     \\
        &         &          \\
\hline
$36$  & $40$  & $76$   \\
\end{tabular}
&

\begin{tabular}{ccc|c}
\multicolumn{4}{c}{\textbf{Size $7$}}\\
\hline
 {\bf 4} & {\bf 5} & {\bf 6}& \textbf{total} \\
\hline
 $5$  & $10$  & $2$  & $17$   \\
 $17$  & $106$  & $95$  & $218$   \\
 $30$  & $180$  &       & $210$   \\
 $51$  &       &       & $51$   \\
\hline
 $103$  & $296$  & $97$  & $496$   \\
\end{tabular}
\end{tabular}
\\
\medskip

\begin{tabular}{c}
\\
\hline
{\bf \# vertices} \\
\hline
\textbf{ $0$ int. pts.} \\
\textbf{ $1$ int. pts.} \\
\textbf{ $2$ int. pts.} \\
\textbf{ $3$ int. pts.} \\
\textbf{ $4$ int. pts.} \\
\hline
\textbf{ total}           \\
\end{tabular}
\hskip -.1cm
\begin{tabular}{|cccc|c}
\multicolumn{5}{c}{\textbf{Size $8$}}\\
\hline
 4 & 5 & 6&7& \textbf{total} \\
\hline
 $5$  & $27$  & $24$  & $3$  &  $59$ \\
 $10$  & $176$  & $393$  & $144$  &  $723$ \\
 $31$  & $429$  & $723$  &       &  $1183$ \\
 $57$  & $563$  &       &       & $620$  \\
 $90$  &       &       &       &  $90$ \\
\hline
 $193$  & $1195$  & $1140$  & $147$   & $2675$  \\
\end{tabular}
\medskip

\begin{tabular}{c}
\\
\hline
{\bf \# vertices} \\
\hline
\textbf{ $0$ int. pts.} \\
\textbf{ $1$ int. pts.} \\
\textbf{ $2$ int. pts.} \\
\textbf{ $3$ int. pts.} \\
\textbf{ $4$ int. pts.} \\
\textbf{ $5$ int. pts.} \\
\hline
\textbf{ total}           \\
\end{tabular}
\hskip -.1cm
\begin{tabular}{|ccccc|c}
\multicolumn{6}{c}{\textbf{Size $9$}}\\
\hline
 4 & 5 & 6&7& 8&\textbf{total} \\
\hline
 $4$  & $43$  & $69$  & $26$  & $1$ & $143$ \\
 $19$  & $195$  & $833$  & $792$  & $151$ & $1990$ \\
 $15$  & $524$  & $2303$  & $1673$  &      & $4515$ \\
 $50$  & $1075$ & $2715$ &       &      & $3840$  \\
 $92$  & $1016$ &       &       &      & $1108$ \\
 $102$  &       &       &       &      &$102$  \\
\hline
$282$  & $2853$ & $5920$ & $2491$  & $152$ &$11698$ \\
\end{tabular}
\\
\medskip

\begin{tabular}{c|cccccc|c}
\multicolumn{8}{c}{\textbf{Size $10$}}\\
\hline
{\bf \# vertices}& 4 & 5 & 6&7& 8&9&\textbf{total} \\
\hline
\textbf{ $0$ int. pts.} & $8$& $56$  & $156$  & $109$  & $16$  & $1$ & $346$ \\
\textbf{ $1$ int. pts.} & $15$& $300$  & $1235$  & $1975$  & $955$  & $107$ & $4587$ \\
\textbf{ $2$ int. pts.} & $21$& $554$  & $3822$  & $6774$  & $2604$  &      & $13775$ \\
\textbf{ $3$ int. pts.} & $37$& $1304$  & $7504$  & $7526$  &       &      & $16371$ \\
\textbf{ $4$ int. pts.} & $92$& $2029$ & $5788$  &       &       &      & $7909$  \\
\textbf{ $5$ int. pts.} & $119$ & $1742$ &       &       &       &      & $1861$ \\
\textbf{ $6$ int. pts.} & $186$ &       &       &       &       &      & $186$  \\
\hline
\textbf{ total} & $478$ & $5985$  & $18505$  & $16384$  & $3575$  & $108$ &$45035$\\
\end{tabular}
\\
\medskip

\begin{tabular}{c|ccccccc|c}
\multicolumn{9}{c}{\textbf{Size $11$}}\\
\hline
{\bf \# vertices}& 4 & 5 & 6&7& 8&9&10&\textbf{total} \\
\hline
\textbf{ $0$ int. pts.} & $6$ & $59$ & $235$  & $267$  & $81$  & $5$ &      & $653$ \\
\textbf{ $1$ int. pts.} & $19$ & $302$ & $1809$  & $3658$  & $2781$  & $748$ & $59$ & $9376$ \\
\textbf{ $2$ int. pts.} & $23$ & $661$ & $5208$  & $13859$  & $12234$  & $2672$ &      & $34657$ \\
\textbf{ $3$ int. pts.} & $32$ & $1326$ & $11892$  & $27467$  & $13474$  &      &      & $54191$ \\
\textbf{ $4$ int. pts.} & $46$ & $2421$ & $16239$  & $19005$  &       &      &      & $37711$  \\
\textbf{ $5$ int. pts.} & $99$ & $3307$ & $12720$  &       &       &      &      & $16126$ \\
\textbf{ $6$ int. pts.} & $185$ & $3356$  &       &       &       &      &      & $3541$  \\
\textbf{ $7$ int. pts.} & $209$ &       &       &       &       &      &      & $209$  \\
\hline
\textbf{ total} & $619$ & $11432$  & $48103$  & $64256$  & $28570$  & $3425$ & $59$ & $156464$\\
\end{tabular}
\caption{The total number of lattice $3$-polytopes of width larger than one and sizes $5$ to $11$, classified according to their numbers of interior lattice points (row) and vertices (column).}
\label{table:interior}
\end{table}

Observe that the minimum volume is not a monotone function of size. There is a (unique) polytope of size $10$ and volume 8, while the minimum volume in size 9 is 10. This may seem contradictory, since for any polytope $P$ of size $n$, the volume of $P^v$, which has size $n-1$, is strictly smaller than that of $P$. However, remember that our table \emph{does not} show the volumes of polytopes of width one (which go from $n-3$ to $\infty$ for each size $n\ge 4$). The polytope of size $10$ and volume $8$ is the second dilation of the unimodular tetrahedron, which is minimal (every proper subpolytope of size one less has width one).

In turn, the maximum volume achieved for each size is very consistent:

\begin{theorem}
\label{thm:largest_vol}
For each $n \in \{5,\dots,11\}$ the maximum volume of a lattice $3$-polytope of size $n$ and width larger than one is $12(n-4)+8$ and is achieved by a unique polytope.
\end{theorem}

Looking closer at these unique polytopes, we can see that there is a very simple description, with only one coordinate of one of the vertices being dependent of the size $n\in\{5,\dots,11\}$. It turns out that this polytope has the same properties for any $n\ge 5$:

\begin{proposition}
\label{pro:biggest_polytope}
The following lattice $3$-polytope 
\[
T_n:=\conv\left\{ (-1,-1,1),(-1,1,-2),(0,1,2n-9),(2,-1,0)\right\}
\]
is a clean tetrahedron of size $n$, width $2$ and normalized volume $12(n-4)+8$, for all $n\ge 5$.
\end{proposition}

\begin{proof}
The width $2$ is achieved with functional $y$. 
Let us see that, besides the four vertices, the only other lattice points of $T_n$ are $n-4$ aligned interior points contained in the line $\ell:=\{x=0=y\}$.
The following is the projection of $T_n$ in the direction of the $z$ coordinate:

\begin{figure}[h]
\centerline{\includegraphics[scale=.8]{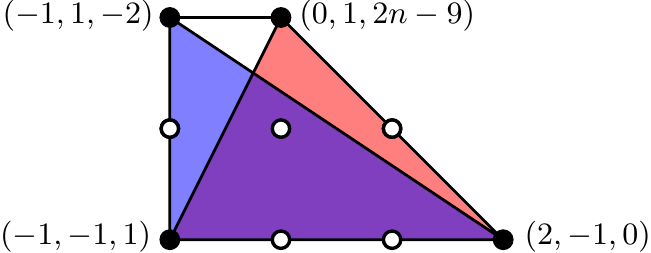}}
\caption{The projection of $T_n$ in the direction of the $z$ coordinate. Black dots represent the projection of vertices, white dots represent other lattice points in the convex hull of the projection, black lines represent edges and the blue and red triangles are the two facets of $T_n$ that intersect the line $\{x=0=y\}$.}
\end{figure}

As the image shows, the only possible lattice points of $T_n$, besides its vertices, could appear as points in the edges $\{(-1,-1,1),(-1,1,-2)\}$, $\{(-1,-1,1),(2,-1,0)\}$ and $\{(0,1,2n-9),(2,-1,0)\}$, or points in the line $\ell$. Since those three edges are primitive (have no interior lattice points) then $T_n$ can only have more lattice points in $\ell$. The figure shows the only two facets that cut this line. The plane passing through $(-1,-1,1)$, $(-1,1,-2)$ and $(2,-1,0)$ cuts $\ell$ at $z=-\frac56 \in (-1,0]$ and the plane passing through $(-1,-1,1)$, $(2,-1,0)$ and $(0,1,2n-9)$ cuts $\ell$ at $z=n-5+\frac56 \in [n-5,n-4)$, for $n-5\ge 0$. Hence the only lattice points of $T_n$ other than its vertices are the points $(0,0,0)$ to $(0,0,n-5)$.\qed
\end{proof}

Given this, it is quite natural to conjecture the following:

\begin{conjecture}
\label{conj:volume}
For each $n\ge5$ the maximum volume of a lattice $3$-polytope of size $n$ and width larger than one is $12(n-4)+8$, and this volume is achieved only by $T_n$.
\end{conjecture}

\begin{remark}
\rm
Han Duong conjectured that the maximum volume of a \emph{clean tetrahedron with exactly $k$ interior lattice points} is $12k+8$, and that the unique clean tetrahedron achieving this bound was $\conv\{(0,0,0),(1,0,0),(0,1,0),(2k+1,4k+3,12k+8)\}$ (Conjecture 2 in~\cite{Han}).
Since that polytope is equivalent to $T_{k+4}$ under the unimodular transformation $(x,y,z)\to (3y-z-1,-2x-2y+z+1,3x+2y-z-2)$, our conjecture is in fact stronger than his: we conjecture that this polytope maximizes volume not only among clean tetrahedra, but actually among all lattice $3$-polytopes of a given size (and width larger than one).
\end{remark}

\subsection{Sublattice index of lattice $3$-polytopes}
\label{sec:nonprimitive}
We call a lattice $d$-polytope $P$ \emph{lattice-spanning} if the lattice spanned by $P\cap \Z^d$ is $\Z^d$. More generally, we call \emph{sublattice index} of $P$ the index, as a sublattice of $\Z^d$, of the affine lattice generated by $P \cap \Z^d$.
 
It is easy to prove (see~\cite{nonprimitive}) that this index coincides with the gcd of all determinants of $(d+1)$-tuples of lattice points in $P$. 
Using this fact we have computed the index of all lattice $3$-polytopes of width larger than one and size up to $11$.
Table~\ref{table:nonprimitive} shows that most polytopes are lattice-spanning and that only indices $1$, $2$, $3$ and $5$ appear.


\begin{table}[h]
\begin{tabular}{l|ccccccc}
{\bf size}      &  {\bf 5}   &  {\bf  6}  &    {\bf 7}  &   {\bf 8}  &   {\bf 9}  &   {\bf 10}  &   {\bf 11}\\
\hline
 {\bf index 1} & 7   &  71   &  486  & 2658  &   11680  &   45012 & 156436  \\
 {\bf index 2} & 0    &   2   &    8   &   14    &   15  &   19  &   24\\
 {\bf index 3} & 1   &   3    &    2  &    3    &   3   &   4    &  4\\
 {\bf index 5} & 1   &   0    &    0   &    0    &   0   &   0    &  0\\
\end{tabular}
\medskip
\caption{Lattice $3$-polytopes of width greater than one and size $\le 11$, classified by sublattice index.}
\label{table:nonprimitive}
\end{table}

In a subsequent paper \cite{nonprimitive} we study the sublattice index of lattice $3$-polytopes of width larger than one and show that, as hinted by Table~\ref{table:nonprimitive}:
\begin{itemize}
\item The only indices that arise are $1$, $2$, $3$ or $5$.
\item There is only one polytope of index $5$, a terminal tetrahedron of normalized volume 20.
\item In size $n\ge 7$ there are exactly $\lfloor n/2\rfloor -1$ polytopes of index three, all closely related to the spiked polytopes of type (4) from Theorem~\ref{theorem:spiked-quasiminimals}.
\item In size $n\ge 9$ there are exactly $\left\lceil {n(n-2)}/4 \right\rceil -1$ polytopes of index two, all closely related to the spiked polytopes of type (1) from Theorem~\ref{theorem:spiked-quasiminimals}.
\end{itemize}

\subsection{Normality in dimension $3$}
\label{sec:normality}

Following~\cite{Bruns_etal}, we say that a lattice $d$-polytope $P$ is \emph{normal} if, for all $k\in \N$, every point in $kP \cap \Z^d$ can be written as the sum of $k$ points in $P\cap \Z^d$. That is, if
\[
kP \cap \Z^d = \left\{ p_1 + \dots +p_ k \; \big\vert \; p_1,\dots,p_k\in P \cap \Z^d\right\},\text{ for all }k\in\N.
\]
Observe that with this definition every normal polytope is lattice-spanning. Sometimes a weaker definition of normality is used, and the concept defined here, which becomes equivalent to ``normal and lattice-spanning'', is called~\emph{integrally closed} (see, e.g.,~\cite{bgBook}).

It is easy to prove that for a lattice $d$-polytope to be normal it is enough that it satisfies the definition for $k\in\{2,\dots, d-1\}$. In particular, a lattice $3$-polytope $P$ is normal if, and only if,
\[
\# \big(2P \cap \Z^3 \big) = \#  \big(
P \cap \Z^3 +P \cap \Z^3 \big).
\]
Via this characterization, we have checked normality in all lattice $3$-polytopes from our database. The resulting numbers are given in the following table, where we also show what fraction of the total are normal, for each size. It is not clear with this data what the asymptotic behavior of this fraction is.
\begin{table}[htb]
\begin{tabular}{r|ccccccc}
size &5&6&7&8&9&10&11\\
\hline
normal &1 & 10 & 61 & 325 & 1532 & 6661 & 25749\\
fraction &0.111 & 0.132 & 0.123 & 0.121 & 0.131 & 0.148 & 0.165\\
\end{tabular}
\label{table:normal}
\end{table}

An interesting question about normality that arises in the work of Bruns et al.~\cite[Question 7.2(a)]{Bruns_etal} is whether, apart from the unimodular tetrahedron, there is a lattice $3$-polytope $P$ that is normal but in which $P^v$ (if it is $3$-dimensional) is not normal for any vertex $v$ of $P$. We can guarantee that, among the $34339$ normal $3$-polytopes of width larger than one and size $\le 11$ there is none.

\subsection{Results on dps $3$-polytopes}
\label{subsec:results-dps}

From our classification it is also easy to extract the full list of dps $3$-polytopes of width greater than one, as shown in Table~\ref{table:dps}.

\begin{table}[h]
\begin{tabular}{c|cccc|c}
{\bf \# vertices}& 4 & 5 & 6&7& \text{total}\\
\hline
{\bf size $5$} & $9$  & $0$  &       &       &$9$\\
{\bf size $6$} & $20$  &$25$   & $0$  &       &$45$\\
{\bf size $7$} & $5$  & $31$  & $12$  & $0$  &$48$\\
{\bf size $8$} & $3$  & $2$  & $1$  &  $0$ & $6$\\
\hline
\bf{total}       & $37$ &$58$ & $13$ &$0$ & $108$
\end{tabular}
\medskip
\caption{The number of dps $3$-polytopes of width larger than one, classified according to their numbers of lattice points (row) and vertices (column).}
\label{table:dps}
\end{table}
  
Those of width one are infinitely many for each given size, but we can also fully classify them since they must consist of two dps polytopes in consecutive planes, and the only dps polytopes in $\R^2$ are: a point, a primitive segment, a unimodular triangle, and the terminal triangle (of volume $3$).
The infinitely many options correspond to the infinitely many possible $GL(\Z,2)$-rotations of one polytope with respect to the other (and infinitely many of those are such that no unimodular parallelogram is in the configuration). The full list for sizes $5$ and $6$ was, moreover, computed in~\cite{5points} and~\cite{6points}.

This completes the classification of dps $3$-polytopes and, in particular, we can answer in dimension $3$ the several questions posed by Reznick~\cite{Reznick-favorite} regarding dps polytopes:
\begin{itemize}
\item \emph{How many ``inequivalent'' dps polytopes of size $2^d$ are there in $\R^d$?
\emph{What is the range for their volume?}
There are six dps $3$-polytopes of width larger than one and of maximal size $8$. Specific coordinates for each of them, together with some other properties, are displayed in Table~\ref{table:dps-size8}. 
The one of minimum volume is a clean $3$-polytope of width $3$ and normalized volume $25$, with $6$ vertices and $2$ interior lattice points. That of maximum volume is a clean tetrahedron of volume $51$.
Five of them (first five rows in the table) were found by Curcic, (unpublished PhD thesis; see~\cite{Curcic}), who asked whether his list was complete. }
\smallskip


\begin{table}[h] \footnotesize
\begin{tabular}{c|c|c|c|c}
\textbf{Coordinates} &\textbf{Vertices}&\textbf{Interior points}&\textbf{Volume}&\textbf{Width}\\
\hline
$\left( \begin{array}{rrrrrrrr}
   -3 &  -1 &   0 &   0 &   0 &   0 &   1 &   1\\
   -3 &  -1 &  -1 &   0 &   0 &   1 &   0 &   4\\
   -1 &   0 &   3 &   0 &   1 &   0 &   0 &  -1
\end{array} \right)$&$4$ &$4$&$51$& $3$ \rule[-5ex]{0pt}{11ex}\\
\hline
$\left( \begin{array}{rrrrrrrr}
   -3 &  -1 &   0 &   0 &   0 &   0 &   1 &   1\\
   -5 &  -1 &   0 &   0 &   1 &   3 &   0 &   8\\
    1 &   0 &   0 &   1 &   0 &  -1 &   0 &  -3
\end{array} \right)$&$4$&$4$&$39$&$3$ \rule[-5ex]{0pt}{11ex}\\
\hline
$\left( \begin{array}{rrrrrrrr}
   -1 &   0 &   0 &   0 &   1 &   1 &   1 &   2\\
   -1 &   0 &   0 &   1 &  -1 &   0 &   2 &  -3\\
   -1 &   0 &   1 &   3 &   0 &   0 &   1 &  -2
\end{array} \right)$&$4$&$4$&$35$&$3$\rule[-5ex]{0pt}{11ex} \\
\hline
$\left( \begin{array}{rrrrrrrr}
   -2 &  -1 &   0 &   0 &   0 &   1 &   1 &   3\\
    1 &  -1 &   0 &   0 &   1 &   0 &   2 &  -1\\
    1 &   0 &   0 &   1 &   0 &   0 &  -1 &   1
\end{array} \right)$&$5$&$2$&$28$&$2$ \rule[-5ex]{0pt}{11ex}\\
\hline
$\left( \begin{array}{rrrrrrrr}
   -2 &  -1 &   0  &  0 &   0 &   1 &   1 &   5 \\
    1 &  -1 &   0  &  0 &   1 &   0 &   2 &  -1\\
    1 &   0 &   0  &  1 &   0 &   0 &  -1 &  -1
\end{array} \right)$&$5$&$3$&$36$\rule[-5ex]{0pt}{11ex}&$2$ \rule[-5ex]{0pt}{11ex}\\
\hline
$\left( \begin{array}{rrrrrrrr}
   -1 &  -1 &   0 &   0 &   0 &   1 &   1 &   2\\
   -2 &  -1 &   0 &   0 &   1 &   0 &   3 &   1\\
   -1 &   2 &   0 &   1 &   1 &   0 &   1 &   0
\end{array} \right)$ &$6$&$2$&$25$&$3$ 
\end{tabular}
\medskip
\caption{Dps $3$-polytopes of size $8$ and width larger than one.}
\label{table:dps-size8}
\end{table}

\item \emph{Is every dps $d$-polytope a subset of one of size $2^d$?
 No. There are exactly $33$ dps $3$-polytopes that have no extension to size $8$. They are all of size $7$ and of width larger than one. Table~\ref{table:dps-maximal} shows the numbers of them, organized according to number of lattice points and vertices:}

\begin{table}[h]
\begin{tabular}{c|cccc|c}
{\bf \# vertices}& 4 & 5 & 6&7& \text{total}\\
\hline
{\bf size $7$} & $3$  & $21$  & $9$  & $0$  &$33$\\
{\bf size $8$} & $3$  & $2$  & $1$  &  $0$ & $6$\\
\end{tabular}
\medskip
\caption{Dps $3$-polytopes of width larger than one that are \emph{maximal} (those that are not contained in another dps $3$-polytope) are counted.}
\label{table:dps-maximal}
\end{table}

\end{itemize}

\vspace{-5ex}

We have also looked at the number of vertices of dps polytopes. In dimension $2$, the maximum number of vertices is $3$. In dimension $3$, and for polytopes of width larger than one, Table~\ref{table:dps} shows that the maximum number of vertices is $6$. The same happens for those of width one, since the lattice dps polygons in each of the parallel planes can only have $3$ vertices each. 

\begin{question}
\label{qn:dps_vertices}
Is the maximum number of vertices of a dps $d$-polytope $3\cdot2^{d-2}$?
\end{question}

Dps polytopes with this number of vertices are easy to construct by induction on the dimension. For $d=2$, dps polygons have $3=3\cdot2^{0}$ vertices. For $d>2$ take two dps $(d-1)$-polytopes with $3\cdot2^{(d-1)-2}=3\cdot2^{d-3}$ vertices and place them in consecutive parallel hyperplanes in a way that no edge in one of the polytopes is parallel to an edge in the other (there are infinitely many possibilities that have this property). Then the resulting polytope is still dps and has $2(3\cdot2^{d-3})=3\cdot2^{d-2}$ vertices.

\bigskip

\end{document}